\pdfoutput=1
\documentclass[11pt,reqno]{amsart}

\usepackage{fullpage}
\usepackage{latexsym}
\usepackage{url}
\usepackage{tikz}
\usepackage{lmodern}
\usetikzlibrary{decorations.pathmorphing} 
\tikzset{snake it/.style={decorate, decoration=snake}}

\usepackage{amsmath}
\usepackage{amsthm}
\usepackage{amssymb}
\usepackage{graphicx}
\usepackage{amsfonts}
\usepackage{hyperref}
\usepackage{bbm}
\usepackage{mathrsfs}
\usepackage{accents}
\usepackage{mathtools}	
\usepackage{color}

\usepackage{mparhack}	
\usepackage{comment}	
\usepackage{marvosym}	
\usepackage{enumerate}	
\usepackage{array}		    
\usepackage{todonotes}

\theoremstyle{plain}
\newtheorem{Prop}{Proposition}[section]
\newtheorem{Thm}[Prop]{Theorem}
\newtheorem*{Thm*}{Theorem}
\newtheorem{Lem}[Prop]{Lemma}
\newtheorem{Cor}[Prop]{Corollary}

\theoremstyle{definition}
\newtheorem{Def}[Prop]{Definition}

\theoremstyle{remark}
\newtheorem{Rem}[Prop]{Remark}

\def\vint_#1{\mathchoice%
          {\mathop{\kern 0.2em\vrule width 0.6em height 0.69678ex
depth -0.58065ex
                  \kern -0.8em \intop}\nolimits_{\kern -0.4em#1}}%
          {\mathop{\kern 0.1em\vrule width 0.5em height 0.69678ex
depth -0.60387ex
                  \kern -0.6em \intop}\nolimits_{#1}}%
          {\mathop{\kern 0.1em\vrule width 0.5em height 0.69678ex
              depth -0.60387ex
                  \kern -0.6em \intop}\nolimits_{#1}}%
          {\mathop{\kern 0.1em\vrule width 0.5em height 0.69678ex
depth -0.60387ex
                  \kern -0.6em \intop}\nolimits_{#1}}}
\def\vintslides_#1{\mathchoice%
          {\mathop{\kern 0.1em\vrule width 0.5em height 0.697ex depth -0.581ex
                  \kern -0.6em \intop}\nolimits_{\kern -0.4em#1}}%
          {\mathop{\kern 0.1em\vrule width 0.3em height 0.697ex depth -0.604ex
                  \kern -0.4em \intop}\nolimits_{#1}}%
          {\mathop{\kern 0.1em\vrule width 0.3em height 0.697ex depth -0.604ex
                  \kern -0.4em \intop}\nolimits_{#1}}%
          {\mathop{\kern 0.1em\vrule width 0.3em height 0.697ex depth -0.604ex
                  \kern -0.4em \intop}\nolimits_{#1}}}

\newcommand{\aveint}[2]{\mathchoice
          {\mathop{\kern 0.2em\vrule width 0.6em height 0.69678ex
depth -0.58065ex
                  \kern -0.8em \intop}\nolimits_{\kern -0.45em#1}^{#2}}%
          {\mathop{\kern 0.1em\vrule width 0.5em height 0.69678ex
depth -0.60387ex
                  \kern -0.6em \intop}\nolimits_{#1}^{#2}}%
          {\mathop{\kern 0.1em\vrule width 0.5em height 0.69678ex
depth -0.60387ex
                  \kern -0.6em \intop}\nolimits_{#1}^{#2}}%
          {\mathop{\kern 0.1em\vrule width 0.5em height 0.69678ex
depth -0.60387ex
                  \kern -0.6em \intop}\nolimits_{#1}^{#2}}}

\DeclareMathOperator{\dv}{div}

\newcommand{\set}[2]{\big\{#1: #2\big\}}
\newcommand{\bigset}[2]{\bigg\{#1: #2\bigg\}}
\newcommand{\mns}{\setminus}
\newcommand{\sdiff}{\triangle}

\newcommand{\N}{\mathbb{N}}
\newcommand{\R}{\mathbb{R}}

\newcommand{\del}{\partial}

\newcommand{\eps}{\varepsilon}
\newcommand{\inv}[1]{{#1}^{-1}}
\newcommand{\dx}{\, dx}
\newcommand{\loc}{\text{\rm loc}}
\newcommand{\Om}{\Omega}
\newcommand{\inp}[2]{\langle\, #1,#2 \,\rangle}
\newcommand{\gr}{\nabla}
\newcommand{\wto}{\rightharpoonup}
\newcommand{\lap}{\Delta}
\newcommand{\D}{\mathscr{D}}

\newcommand{\M}{\mathcal{M}}
\newcommand{\A}{\mathscr{A}}
\newcommand{\Rsp}[1]{\mathscr{R}_{p,\, #1}}
\newcommand{\gto}{\xrightarrow{\gamma_p}}
\newcommand{\Gto}{\xrightarrow{\Gamma_p}}
\newcommand{\pcap}{\text{Cap}_p}

\title[The second eigenvalue of the $p\,$-Laplacian]
 {A variational characterisation of the second eigenvalue of the $p\,$-Laplacian on quasi open sets} 

\author[Nicola Fusco]{Nicola Fusco}
\address[Nicola Fusco]{Department of Mathematics and Application, University of Naples `Federico II'}
\email{n.fusco@unina.it}

\author[Shirsho Mukherjee]{Shirsho Mukherjee}
\address[Shirsho Mukherjee]{Department of Mathematics and Statistics, University of Jyv\"askyl\"a}
\email{shirsho.s.mukherjee@jyu.fi}

\author[Yi Ru-Ya Zhang]{Yi Ru-Ya Zhang}
\address[Yi Ru-Ya Zhang]{Hausdorff Centre for Mathematics, University of Bonn}
\email{yizhang@math.uni-bonn.de}


\begin{document}
\maketitle

\begin{abstract}
In this article, we prove a minimax characterisation of the second eigenvalue of 
the $p$-Laplacian operator on $p$-quasi open sets, using a construction based on minimizing movements. This leads also to an  
existence theorem for spectral functionals depending on the first 
two eigenvalues of the $p$-Laplacian. 
\end{abstract}


\section{Introduction}\label{sec:introduction}

The Dirichlet eigenvalues of the $p$-Laplacian operator are defined as the  numbers $\lambda>0$ for which the following Dirichlet problem  
\begin{equation}\label{eq:plap eigen} 
\begin{cases}
-\dv(|\gr u|^{p-2}\gr u)  = \lambda |u|^{p-2}u \ \ &\text{in}\  \Om \\
 \qquad \qquad \qquad \quad\  u =0 \ \ &\text{in}\ \del\Om, 
\end{cases}
\end{equation}
admits a 
non-zero weak solution $u\in W^{1,\,p}_0(\Om)$. Here 
$\Om\subset \R^n$ is an open set of finite measure and $1<p<\infty$. In fact, the 
eigenvalues are the critical values of the Rayleigh quotient 
$$
\mathcal R_\Om(w)=\frac{\displaystyle\int_\Om |\gr w|^p\dx}{\displaystyle\int_\Om |w|^p\dx}
$$
and the corresponding weak solutions of equation \eqref{eq:plap eigen} are 
the critical points of $\mathcal R_\Om(w)$ 
among all non-zero functions in $w\in W^{1,\,p}_0(\Om)$. While the 
first eigenvalue $\lambda_1(\Om)$ is defined as the minimum value of 
$\mathcal R_\Om(w)$, not much is known on higher order eigenvalues when $p\not=2$. One way to obtain them is by using the so called {\it Krasnoselskii's genus} $\gamma(\M)$ of a 
set $\M \subset W^{1,\,p}_0(\Om)$. In fact, it was shown in \cite{Gar-Per} that, 
denoting by $\Sigma_k,\ k= 1,2,\ldots$, the collection of all symmetric subsets 
$\M$ contained in $W^{1,\,p}_0(\Om)$ with $\gamma(\M) \geq k$, the 
numbers
\begin{equation}\label{eq:genk}
\lambda_k(\Om)=\inf_{\M\in \Sigma_k} \Big[\sup_{u\in\M} \mathcal R_\Om(w)\Big]
\end{equation}
form an increasing sequence of eigenvalues. It is not known whether all the 
eigenvalues of the $p$-Laplacian are of this form if $p\not=2$. However, it was proved by 
Anane-Tsouli \cite{An-Ts} that, given a bounded and connected open set $\Om$, if 
$\lambda_1(\Om)$ and $\lambda_2(\Om)$ are defined as in \eqref{eq:genk}, then 
$\lambda_1(\Om)$ is the smallest eigenvalue and there are no other eigenvalues 
in the interval $(\lambda_1(\Om),\lambda_2(\Om))$, see also \cite{Juu-Lind}.

Another variational characterisation of $\lambda_2(\Om)$ was given by 
Cuesta-de Figueiredo-Gossez in \cite{Cue-Fig-Gos}, who proved that for a bounded open and connected set $\Om$ we have 
\begin{equation}\label{eq:var char}
  \lambda_2(\Om)\, =\inf_{\gamma\in\Gamma(u_1,-u_1)}
  \bigg[\,\max_{w\in\gamma([0,1])} \int_\Omega |\gr w|^p \dx\,\bigg], 
 \end{equation}
where $u_1$ is the first nonnegative eigenfunction with 
$\|u_1\|_{L^p(\Om)}=1$ and $\Gamma(u_1,-u_1)$ is the family of all continuous 
maps from $[0,1]$ to $\M_p(\Om) = \set{u\in W^{1,\,p}_0(\Om)}{\|u\|_{L^p(\Om)}=1}$ with endpoints $u_1$ and~$-u_1$. Later on, it was shown by  
Brasco-Franzina \cite{Bra-Fra--hks} that \eqref{eq:var char} still holds if $\Om$ is any open set of finite measure, not necessarily connected. Finally, as pointed to us by L. Brasco, a different variational characterization of  $\lambda_2(\Om)$ can be obtained by combining a result proved in \cite{DR} with the argument used by Brasco and Franzina in the proof of \cite[Th. 4.2]{BF}:
\begin{equation}\label{varchar3}
\lambda_2(\Omega)=\inf_{f\in{\mathcal C}_{odd}(\mathbb{S}^1,\M_p(\Om))}  \bigg[\,\max_{u\in Im(f)}\int_\Omega |\gr w|^p \dx\,\bigg],
\end{equation}
where ${\mathcal C}_{odd}(\mathbb{S}^1,\M_p(\Om))$ is the set of continuous and odd maps from $\mathbb{S}^1$ to $\M_p(\Om)$.

In this paper, we study the properties of the first two eigenvalues of the 
$p$-Laplacian in a $p$-quasi open set $A$. Beside being of its own interest, this study is motivated by the existence Theorem~\ref{thm:mainthm2} below. Recall that $A\subset\R^n$ is 
$p$-quasi open if there exists a $p$-quasi continuous nonnegative function $u\in W^{1,\,p}(\R^n)$ such that 
$A = \{u>0\}$ (see Section \ref{sec:qfopen} for more details). In order to deal with the $p$-Laplacian on $p$-quasi open sets we need to introduce the $p$-fine topology, which turns out to be an important tool for our study. In particular some basic properties of Sobolev functions on quasi open sets are given in Theorem~\ref{thBB} and Lemmas~\ref{lem:quasi res} and \ref{sob}, while the strong minimum principle for the $p$-Laplacian on $p$-quasi open sets is stated in Theorem~\ref{thm:strong min}.

For bounded 
open sets, it is known (see \cite{Lind2}) that if the first eigenvalue is simple, then it is isolated. Here we prove that the same holds in the 
framework of $p$-quasi open sets, see Proposition~\ref{prop: first eigen prop}. This fact  turns out to be useful in the proof of our main result, which states that  formula \eqref{eq:var char} is still true when $\Om$ is replaced by a $p$-quasi open set $A$ of finite measure.
\begin{Thm}\label{thm:mainthm1}
Let $A\subset\R^n$ be a $p$-quasi open set of finite measure, let $u_1 \in W^{1,\,p}_0(A)$ be a normalized 
 eigenfunction of $\lambda_1(A)$ and let  
 $\Gamma(u_1,-u_1)
 =\set{\gamma\in C\big([0,1],\M_p(A)\big)}{\gamma(0) = u_1,\ \gamma(1) = -u_1}$. 
 Then
 \begin{equation}\label{thm11}
  \lambda_2(A)\, =\min_{\gamma\in\Gamma(u_1,-u_1)}
  \bigg[\,\max_{w\in\gamma([0,1])} \int_A |\gr w|^p \dx\,\bigg]. 
 \end{equation}
   \end{Thm}
   
Note that, differently from what was known before in the case of open sets, see \eqref{eq:var char} and  \eqref{varchar3}, we prove here that the infimum at the right hand side of \eqref{thm11} is indeed attained.

A few words on the proof of the minimax formula \eqref{thm11} for a $p$-quasi open set $A$ are in order. Assuming that 
$\lambda_1(A)$ is simple, the idea is to show the existence of a   curve
$\gamma :[0,1] \to \M_p(A)$ connecting $u_1$ and $-u_1$ such that 
$$ \max_{w\in\gamma([0,1])} \int_A |\gr w|^p \dx =\lambda_2(A).$$
In the case of a bounded connected open set, the construction of this path $\gamma$ in \cite{Cue-Fig-Gos} involves a delicate use of Ekeland's variational 
principle which does not seem to work for quasi open sets. 

Therefore we have 
chosen here a completely different approach, based on De Giorgi's 
{\it minimizing movements}. 
Indeed, we construct the desired path by joining three different curves. One of them, connecting the negative and positive parts of a second eigenfunction $u_2$  is easily constructed by hands. The construction of the two other curves is where we use the minimizing movements, see Lemma~\ref{lem:path}.
Precisely, we consider the limit of a sequence of maps $v_h :A\times [0,\infty) \to \M_p(A)$, where $v_h$ are the gradient flows of the $p$-energy functional 
$$ E(u) =\int_\Om |\gr u|^p\dx $$
restricted to the manifold $\M_p(A)$, with respect to the $L^p(A)$-distance 
in $W^{1,\,p}_0(A)$. The maps $v_h$ are  time-discretized weak solutions of 
the following doubly nonlinear evolution equation
 \begin{equation}\label{eq:evol1}
 \begin{dcases*}
  \ |\del_t u|^{p-2}\del_t u=\dv(|\gr u|^{p-2}\gr u)
  + \sigma(t)|u|^{p-2}u\qquad &  in $ A\times (0,\infty)$ \\
 u \in W^{1,\,p}_0(A)\cap \M_p(A) & for all $ t\geq 0$
 \end{dcases*}
\end{equation}
with $u(0) = v_0$, where $\lambda_1(A) < E(v_0) \leq \lambda_2(A)$ and 
$v_0$ is either $u_2^+$ or $-u_2^-$. It turns out that for 
 every $h$, the energy functional $E(v_h(t))$ is strictly decreasing along the flow for $t>0$ and we have $E(v_h(t)) \to \lambda_1(A)$ and $v_h(t) \to u_1$ (or $-u_1$) in $W^{1,\,p}_0(A)$ as $t\to +\infty$. Then we show that the flows $v_h$ converge to 
 a map $v:A\times [0,\infty) \to \M_p(A)$ weakly in 
 $W^{1,\,p}((0,\infty), L^p(A))$ and strongly in $W^{1,\,p}_0(A)$ for almost every 
 $t\geq 0$, as $h\to +\infty$. Although these convergences do not imply that 
 $v$ is also a weak solution of the equation \eqref{eq:evol1}, they are enough to 
 conclude that 
 $$ E(v(t)) \, < \, \lambda_2(A)\ \ \forall\ t>0\quad 
 \text{and}\quad v(t) \to u_1\ (\text{or} -u_1)\ \text{as}\ t\to+\infty. $$

An immediate application  of the variational characterisation \eqref{thm11} of the second eigenvalue of the $p$-Laplacian   is the lower semicontinuity of 
$\lambda_2(A)$ with respect to a suitable convergence in the 
family of all $p$-quasi open subsets of a bounded open set $\Om$. In turn, 
this lower semicontinuity leads to the following existence theorem,  
where we denote by $\A_p(\Om)$ the family of all $p$-quasi open sets contained in $\Omega$.
\begin{Thm}\label{thm:mainthm2}
 Let $ f: \R^2 \to \R $ be a lower semicontinuous function, separately increasing in both variables,  and let $\Omega\subset\R^n$ be a bounded open set.  
 For every $\ 0< c \leq |\Om| $, there exists a $p$-quasi open minimizer of the   following problem 
$$
  \min\set{f\left(\lambda_1(A),\lambda_2(A)\right)}{A\in \A_p(\Om),\ |A|=c}.
$$
\end{Thm}
Throughout the paper we shall always assume that $1<p\leq n$, unless otherwise stated. If $p> n$  then $p$-quasi open sets reduce to open sets for which all the results contained in Sections 2 and 3 are well known. On the other hand, the results proved in Sections 4 and 5, which are new also in the context of open sets, are proved exactly in the same way regardless of the fact that $p$ is smaller or greater than $n$, see Remark~\ref{rem5}.

Finally we would like to thank the anonymous referees for the valuable comments and for suggesting us a few corrections and simplifications in the proofs.

\section{Quasi open and finely open sets}\label{sec:qfopen}

In this section we shall review the notions of $p$-capacity and $p$-quasi open sets and prove some results that will be crucial for the rest of the paper.  However, since quasi open sets do not form a topology, we  need to introduce also the related notion of 
$p$-finely open sets, which do form a topology. For all the main properties and the basic results  needed in the sequel and not proven here, we refer to \cite{Fuglede,Fuglede--quasitop,Adams-Lewis,HKM1990,Kilp-Mal} and the references therein.  \par
Finally, we warn the reader that sometimes we shall drop the notation $p$ whenever it is  clear from the context that we refer  to 
$p$-quasi or $p$-finely open sets.

Given a measurable set $E\subset\R^n$, we define its {\it $p$-capacity} by setting 
$$
   \pcap(E):= \inf\bigset{\int_{\R^n}(|u|^p+ |\gr u|^p)\dx}{u\in W^{1,\,p}(\R^n), \ u \geq 1 \
   \text{a.e. in an open set}\ U \supset E}.
$$
 Note that this definition is equivalent to the Bessel capacity $C_{1,\,p}$ defined via the Bessel kernel $G_1$ (see \cite[Rem. 1.13]{Kilp-Mal} or \cite[Sect. 2.6]{Ziemer} for more 
details). 

If a property holds everywhere except possibly in a set of zero $p$-capacity, we say that it holds {\it $p$-quasi everywhere} (and we write {\it $p$-q.e.} or {\it q.e.} for short).

\begin{Def}[(Quasi open sets)]\label{def:qopen}
{\rm A  set $A\subset\R^n$ is said to be $p$-{\it quasi open} if for every $ \eps >0 $ there exists an open set 
 $U_\eps$ such that $ \pcap(U_\eps\sdiff A) < \eps$; equivalently,  if there exists an open set $A_\eps$ such that $A \cup A_\eps$ is open and 
 $\pcap(A_\eps) < \eps$. }
\end{Def}
A function $f:A\to\R$ defined on a quasi open set $A$ is said to be {\it $p$-quasi continuous} if 
for every $\eps >0$, there exists an open set $A_\eps$ such that  $\pcap(A_\eps) < \eps$ and the restriction of $f$ to $A\setminus A_\eps$ is continuous. More equivalent definitions are contained in Theorem~\ref{thm:quasi fine} below.
\medskip

It is well known that any function  $u\in W^{1,\,p}(\R^n)$ has a $p$-quasi continuous representative   $v$. In particular, $u=v$ a.e. in $\R^n$. Moreover $v$ is unique in the sense that if $w$ is another $p$-quasi continuous representative of $u$, then $v=w$ q.e., see \cite[Th. 1.3]{Kilp-Mal}. Henceforth, when dealing with a function in $W^{1,\,p}(\R^n)$, we shall always assume that $u$ is $p$-quasi continuous. 
\begin{Def}\label{def:$p$-quasi sobolev}
{\rm
For an open set $\Omega\subset \R^n$ we 
denote the collection of all $p$-quasi open subsets of $\Omega$ by $\A_p(\Omega)$.
If $A$ is $p$-quasi open and $\pcap(A) > 0$ we say  that $u\in W^{1,\,p}(\R^n)$ belongs to the space  $W^{1,\,p}_0(A)$ if  any $p$-quasi continuous representative of $u$ vanishes $p$-q.e. in $\R^n\setminus A$. The space $W^{1,\,p}_0(A)$, equipped with the norm naturally induced by $W^{1,\,p}(\R^n)$, is a Banach space. Setting $p^\prime=p/(p-1)$, we denote by $W^{-1,\,p^\prime}\!(A)$  the dual space of $W^{1,\,p}_0(A)$. }
 \end{Def}
Note that  the space $W^{1,\,p}_0(A)$ can be equivalently defined by setting
$$
W^{1,\,p}_0(A)=\bigcap\set{W^{1,\,p}_0(U)}{U\,\,\text{open,}\,\,U\supset A},
$$
see \cite[Th. 2.10]{Kilp-Mal}. However, we shall never use this characterisation in the sequel.
\medskip

Since $p$-quasi open sets do not form a topology, we introduce the {\it $p$-fine topology} which is the coarsest topology on $\R^n$ making all (classical) $p$-superharmonic 
functions continuous. A more robust equivalent definition can be given using 
the Wiener criteria, as follows.
\begin{Def}[(Finely open sets)]\label{def:fine}
{\rm 
 A set $U\subset \mathbb R^n$ is  {\it $p$-finely open}  if for every
 $x\in U$ 
 $$
  \int_{0}^{1}
  \bigg( \frac{\pcap(B_r(x)\setminus U)}
  {r^{n-p}}\bigg)^{\frac 1 {p-1}} 
 \,\frac {dr} r< \infty.
$$
}
\end{Def}
The fine topology has been extensively studied in the context of nonlinear potential theory. 
For more details we refer the reader to 
\cite{Fuglede}, \cite{Fuglede--quasitop}, \cite{Adams-Lewis}, \cite{Kilp-Mal} and to the references therein. 
We recall here the following result, see  \cite[Th. 1.4 and 1.5]{Kilp-Mal}, 
which deals with the compatibility of finely open sets with 
quasi-open sets. 
\begin{Thm}\label{thm:quasi fine}
 Given a set $A\subset\R^n$, the following are equivalent.
 \begin{enumerate}[(i)]
  \item $A$ is $p$-quasi open.
  \item $A =  U \cup E$ where $ U$ is $p$-finely open and $\pcap(E) = 0$.
  \item There exists a $p$-quasi continuous function $u\in W^{1,\,p}(\R^n)$, $u\geq0$, such that $A = \{u>0\}$.
 \end{enumerate}
 Furthermore, given a $p$-quasi open set $A$ and a function $f:A\to\R$, the following are equivalent. 
 \begin{enumerate}[(i)]
  \item $f$ is $p$-quasi continuous in $A$.
  \item The sets $\{f>c\}$ and $\{f<c\}$ are $p$-quasi open for all $c\in\R$.
  \item $f$ is $p$-finely continuous in $A$ up to a set of zero $p$-capacity.
 \end{enumerate}
\end{Thm}  
\begin{Rem}\label{rem:q-unique}
{\rm
From Definition \ref{def:qopen} it is immediate that 
 a $p$-quasi open set $A$ remains quasi open if we change it by a 
 set of zero $p$-capacity. This makes the characterisation 
 $A =  U \cup E$  in Theorem~\ref{thm:quasi fine} unique up to sets of 
 zero $p$-capacity. On the other hand, if $U$ is $p$-finely open and $E$ has zero $p$-capacity, it is easy to check from Definition~\ref{def:fine} that also $U\setminus E$ is $p$-finely open. Note also that if $U$ is $p$-finely open and $\pcap(U)=0$ then $U$ is empty.
 }
 \end{Rem}
We  need to introduce also the notion of quasi connectedness, as follows. 
\begin{Def}[(Quasi connected sets)]\label{def:qconn}
{\rm
 A $p$-quasi open set $A\subset \R^n$ is  $p$-{\it quasi connected} if for any $p$-quasi open sets $A_1$, $A_2$ such that $A=A_1\cup A_2$ and $\pcap(A_1\cap A_2)=0$, then either $\pcap(A_1)=0$ or $\pcap(A_2)=0$. 
 }
\end{Def}

The notion of $p$-quasi connectedness is closely related to the topological notion of $p$-finely connected set. Indeed a much stronger result holds, due to 
A. Bj\"orn-J. Bj\"orn \cite[Th. 1.1]{BB}. 

\begin{Thm}\label{thBB}
Let $A$ be a $p$-quasi open set. Then the following are equivalent.
\begin{enumerate}[(i)]
  \item If $u\in W^{1,\,p}_\loc(\R^n)$ and $\nabla u=0$ a.e. in $A$, then there exists a constant $c$ such that $u=c$ a.e. in $A$.
  \item $A$ is $p$-quasi connected.
  \item $A=U\cup E$, where $U$ is $p$-finely connected and $p$-finely open  and $\pcap(E)=0$.
  \end{enumerate}
\end{Thm}
The equivalence between $(ii)$ and $(iii)$ in the above theorem is a straightforward consequence of the definition and of Theorem~\ref{latvala} below, while the equivalence between these two conditions and $(i)$ is the main result in \cite{BB}.

Next result, due to Latvala \cite[Th. 1.1]{Latvala} will be used later in this section.
 
 \begin{Thm}\label{latvala}
 Let $U$ be open and connected in the $p$-fine topology and let $E$ be a set of zero $p$-capacity. Then $U\setminus E$ is also open and connected in the $p$-fine topology.
  \end{Thm}
It is known that the $p$-fine topology has the {\it quasi Lindel\"of property}, 
i.e. every family $\{U_\alpha\}_{\alpha\in\mathcal A}$ of 
$p$-finely open sets contains an at most  countable subfamily 
$\{U_h\}$ such that $\cup_{\alpha\in\mathcal A}U_\alpha=\cup_hU_h$ up to a set of zero $p$-capacity (see \cite[Sect.~12]{Fuglede} for the case $p=2$ and \cite{HKM1990} 
for  $p\neq 2$). Moreover, it is also known that the $p$-fine topology is locally connected, see \cite[Th. 3.15]{HKM1990}. Using these properties it is straightforward to check that  a  $p$-quasi open set $A$ can be always decomposed as 
\begin{equation}\label{eq:num}
A = \bigcup_{j\in\N}U_j\cup E,
\end{equation}
where  the $U_j$ are   $p$-finely open and  $p$-finely connected sets,  pairwise disjoint, and $\pcap(E)=0$.  We shall refer to the sets $U_j$ as to
the {\it $p$-quasi connected components} of $A$. Note that  
they are uniquely determined up to a set of zero $p$-capacity,  
as explained in Remark \ref{rem:q-unique}. 
\medskip

The next lemma deals with the restrictions of a function in $W^{1,\,p}_0(A)$  
on the $p$-quasi connected components of a quasi open set $A$. To this aim, given a set $E\subset\R^n$, for every $x'\in\R^{n-1}$ we set  $E_{x^\prime}=\set{t\in\R}{(x',t)\in E}$ and we shall denote by $\mathcal H^{n-1}$ the $(n-1)$-dimensional Hausdorff measure.

\begin{Lem}\label{lem:quasi res}
Let $A$ be a  $p$-quasi open set of finite measure 
and let $V \subset A$ be a $p$-quasi connected component of $A$. 
For any  $u\in W^{1,\,p}_{0}(A)$, 
we have $u_{|{V}}\in W^{1,\,p}_{0}(V)$. 
\end{Lem}
\begin{proof}
Thanks to \eqref{eq:num}, we may write $A = V\cup U \cup E$, where $V, U, E$ are mutually disjoint, $V$ is as in the statement, $U$ is $p$-finely open and  $\pcap(E)=0$.
 Let $u\in W^{1,p}(\R)$ be a $p$-quasi continuous function, $u=0$ q.e. in $\R^n\setminus A$ and  set
 $$
v(x):\ =
\ \begin{cases}
u(x)&\text{ if }x\in V,\\
0&\text{elsewhere}.
\end{cases}
$$
Since $u$ is quasi continuous in $\R^n$, by Theorem~\ref{thm:quasi fine} it coincides with a finely continuous function up to a set of zero $p$-capacity. Then, it is easily checked that  also  $v$ is finely continuous up to a set of zero $p$-capacity and thus quasi continuous in $\R^n$.  
We claim that for $\mathcal L^{n-1}$-a.e. $x'\in\R^{n-1}$ the function $v(x',\cdot)$ is in $W^{1,\,p}(\R)$.

To this end recall that   $\pcap(E)=0$, hence $\mathcal H^{n-1}(E)=0$, see \cite[Sect. 4.7.2]{Evans-Gar}. Using this fact and  Corollary~\ref{cor:addendum}, we get that there exists a set $Z_0\subset\R^{n-1}$ with $\mathcal L^{n-1}(Z_0)=0$, such that for $x'\not\in Z_0$   the sections $V_{x^\prime}$ and $U_{x^\prime}$ are open  and  $E_{x^\prime}=\emptyset$.  

Since the {\it precise representative} $u^*$ of $u$ is $p$-quasi continuous, see \cite[Th. 3.10.2]{Ziemer},  with no loss of generality  we may take $u=u^*$. Thus we may  assume that $u(x',\cdot)\in W^{1,\,p}(\R)\cap C(\R)$, see \cite[Sect. 4.9.1]{Evans-Gar}. In addition, since $u(x)=0$ q.e. in $\R^n\setminus A$, enlarging $Z_0$ if needed, we may also assume that  $u(x',\cdot)\equiv0$ on the closed set $\R\setminus A_{x^\prime}$ for all $x'\not\in Z_0$. Since  the sections $V_{x^\prime}, U_{x^\prime}$ form a partition of $A_{x^\prime}$  and $u(x',\cdot)\in W^{1,\,p}_0(A_{x^\prime})$, we conclude that $u(x',\cdot)\in W^{1,\,p}_0(V_{x^\prime})$. This proves that $v(x',\cdot)\in W^{1,\,p}(\R)$, as claimed. 

Note that
$$
\int_{\R^{n-1}}\bigg(\int_\R\Big|\frac{d}{dt}v(x',t)\Big|^p\,dt\bigg)\,dx'\,\leq\,\int_{A}|\nabla u|^p\,dx.
$$
Repeating the above argument in all coordinate directions we conclude that $v\in W^{1,\,p}(\R^n)$. In fact, $v\in W^{1,\,p}_0(V)$ since it is $p$-quasi continuous and vanishes  q.e. outside $V$. This completes the proof in this case.
\end{proof}
The next lemma will be used in the proof of Proposition~\ref{extAH}.
\begin{Lem}\label{sob}
Let $A$ be a  $p$-quasi open and $p$-quasi connected set. Let $u,v\in W^{1,\,p}_0(A)$ be two  functions such that $u,v>0$ q.e. in $A$ and 
$$
\frac{\gr u}{u}=\frac{\gr v}{v} \qquad \text{a.e. in $A$.}
$$
Then, there exists a constant $\kappa>0$ such that $u=\kappa v$ q.e. in $A$.
\end{Lem}
\begin{proof} By subtracting from $A$ a set of zero $p$-capacity and using Theorem~\ref{latvala}, we may assume that $A$ is $p$-finely open and $p$-finely connected, that $u$ and $v$ are $p$-finely continuous in $A$ and that $u(x),v(x)>0$ for all $x\in A$. We claim that for any $x\in A$ there exists a $p$-finely open neighborhood of $x$ where $u/v$ is constant. Then the result will follow immediately.

To prove this claim, fix $x\in A$. Recall that the $p$-fine topology is locally connected. Thus we may find a $p$-finely connected neighborhood $V_x$ of $x$ contained in the $p$-finely open set $A\cap\{u>\eps\}\cap\{v>\eps\}$, where $0<\eps<\min\{u(x),v(x)\}$. Since  $w_\eps=\log\max\{u,\eps\}-\log\max\{v,\eps\}\in W^{1,p}_{loc}(\R^n)$ and $\nabla w_\eps=0$ in $V_x$, Theorem~\ref{thBB} yields that $u/v$ is constant on $V_x$. This proves the claim, thus concluding the proof of the lemma.
\end{proof}

\section{Eigenvalues of the $p$-Laplacian in a quasi open set}\label{sec:eigenp}
 In this section we study the main properties of the first and second eigenvalue of the $p$-Laplacian in a $p$-quasi open set and establish the variational characterisation of Theorem \ref{thm:mainthm1}. This will be used in Section \ref{sec:gamma cont} to establish the lower semicontinuity of $\lambda_2(A)$.  
 
 \subsubsection{The $p$-Laplacian and the Resolvent}\noindent		
\\
Given a quasi open set $A\in \A_p(\Omega)$ and
$f\in  W^{-1,\,p^\prime}(A)$, the Dirichlet problem
\begin{equation}\label{eq:p Laplacian}
\begin{cases}
-\dv(|\gr u|^{p-2}\gr u)  =\ f \ \ \text{in}\ \ A \\
 u \in W^{1,\,p}_0(A)
\end{cases} 
\end{equation}
is defined in the usual weak sense. Precisely, we say that $u\in W^{1,\,p}_0(A)$ is a {\it weak solution} of the equation \eqref{eq:p Laplacian} if  
for every $\phi\in W^{1,\,p}_0(A)$, we have 
 $$ \int_A |\gr u|^{p-2}\gr u\cdot\gr \phi\ dx =\inp{f}{\phi}, $$   
where $\inp{\cdot\,}{\cdot}$ denotes the duality pairing between $W^{-1,\,p^\prime}\!(A)$ and $W^{1,\,p}_0(A)$.

Let us set $\lap_pu:= \dv(|\gr u|^{p-2}\gr u) \in W^{-1,\,p^\prime}\!(A)$. Following \cite{Kilp-Mal}, we say that a $p$-quasi continuous function $u\in W^{1,\,p}(A)$ is a {\it fine supersolution} of the equation $-\lap_pu= 0$, if for every nonnegative function $\phi\in W^{1,\,p}_0(A)$, we have 
 $$ \int_A |\gr u|^{p-2}\gr u\cdot\gr \phi\ dx\geq0.$$ 
The monotonicity of the $p$-Laplacian operator ensures, as in the standard case of an 
open set, the existence of a unique weak solution of 
\eqref{eq:p Laplacian}. This enable us to define the  resolvent map as usual.
\begin{Def}[(Resolvent)]\label{def:res}
{\rm
For a quasi open set $A\in \A_p(\Om)$, the  
{\it resolvent map} for  the $p$-Laplacian operator, is defined  for any $f\in W^{-1,\,p^\prime}\!(A)$ by setting $\Rsp{A} (f) := u $, where  $ u \in W^{1,\,p}_0(A)$ is the unique weak solution 
of \eqref{eq:p Laplacian}. 
}
\end{Def}
 We recall from Kilpel{\"a}inen-Mal{\'y} \cite{Kilp-Mal} the following theorem related to fine supersolutions. 
\begin{Thm}[({\cite[Th. 4.3]{Kilp-Mal}})]\label{thm:pharnack}
 Let  $A\subset \R^n$ be $p$-quasi open and let $u_j\in W^{1,\,p}(A)$ be an increasing sequence of fine supersolutions of 
converging q.e. to a function $u:A\to\R$.
Then $u$ is $p$-quasi continuous. 

\end{Thm}
An immediate consequence of the above theorem 
is the following minimum principle for fine supersolutions, see \cite[Th. 4.1]{Latvala}. We give its simple proof for the reader's convenience.
\begin{Thm}[(Minimum principle)]\label{thm:strong min}
 Let $u\in W_0^{1,\,p}(A)$ be a fine
supersolution on a $p$-quasi open and quasi connected set $A\subset \mathbb R^n$, $u\geq0$ q.e.. Then  either $u>0$  or $u=0$ q.e. in~$A$.
\end{Thm}
\begin{proof}
First, we recall that the minimum of two fine supersolutions is also a fine supersolution, see \cite[Prop. 3.5]{Kilp-Mal}. Therefore, given 
$u\in W_0^{1,\,p}(A)$ as in the statement, the functions $v_j:=\min\{ju,\,1\}$, $j\in\mathbb N$, form an increasing sequence of fine supersolutions converging q.e. in $A$ to the function
 $v$  given by 
\begin{equation*}
 v(x) =
 \begin{cases}
  0\ \qquad &\text{if}\quad x \in \{u=0\}, \\
  1\ \qquad &\text{if}\quad x \in \{u>0\}.
 \end{cases}
\end{equation*}
By Theorem~\ref{thm:pharnack}, $v$ is $p$-quasi continuous in $A$, hence by Theorem~\ref{thm:quasi fine} the sets $\{v<1\} =\{u=0\}$ and $\{v>0\}=\{u>0\}$ are both quasi open. Then the conclusion  follows immediately from the assumption that $A$ is $p$-quasi connected.
\end{proof}

\subsection{Variational eigenvalues}\label{subsec:var eigen}\noindent
We now recall for the reader's 
convenience a few well known facts from the 
classical variational theory of eigenvalues which 
we will need later. We refer the reader to 
\cite{G} for more details.
\subsubsection{Minimax characterisation}\noindent
\\
A useful way to deal with nonlinear eigenvalues is to obtain them via the Euler-Lagrange equation of constrained minimization problems. Given a Banach space $X$ and two
functionals $ E,G \in C^1(X) $, the minimizers $u$ of the constrained problem  $\min\set{E(w)}{G(w) =1}$ satisfy the equation 
\begin{equation}\label{eq:constrained problem}
DE(u) =\lambda DG(u)
\end{equation}
for some Lagrange multiplier $\lambda \in \R$, where $D$ denotes the
Fr\'echet derivative. There are several ways of 
generating such constrained critical values $\lambda$, if the functional $E$ is invariant with respect to 
some compact group of symmetries acting on the manifold 
 \begin{equation}\label{mani}
 \M:= \set{w\in X}{G(w) =1}.
 \end{equation}
Here  we use a variant of the {\it mountain pass lemma} 
by Ambrosetti-Rabinowitz \cite{Amb-Rab}. Recall that the norm of the Fr\'echet  derivative of the restriction $\widetilde E$ of $E$ to $\mathcal M$   
at a point $u\in\mathcal M$, is defined as
$$
\|D\widetilde E(u)\|_{*}:= \min\set{\big\|\,DE(u)-tDG(u)\big\|_{X^*}}{t\in\R},
$$
where $\|\cdot\|	_{X^*}$ denotes the norm of the dual space $X^*$. It is said  that the functional $E$ satisfies the {\it Palais-Smale condition on} $\mathcal M$, if for any sequence $u_h\in\mathcal M$ such that $E(u_h)$ is bounded and $\|D\widetilde E(u_h)\|_{*}\to0$, there exists a subsequence of $u_h$ converging strongly in $X$. Then, the following result holds, see \cite[Prop. 2.5]{Cue-Fig-Gos}, \cite[Th. 3.2]{G}.
\begin{Thm}\label{thm:Mpass}
Let $X$ be a Banach space and $E, G\in C^1(X)$. Let $\mathcal M$ be as in \eqref{mani} and assume that $DG\not=0$ on $\mathcal M$ and that $E$ satisfies the Palais-Smale 
 condition on $\mathcal M$. 
 \par
 Let $u_0,u_1\in\mathcal M$ and $\rho>0$ be such that $\|u_1-u_0\|_X>\rho$ and
 $$
 \inf\set{E(u)}{u\in\mathcal M,\, \|u-u_0\|_X=\rho}>\max\{E(u_0),E(u_1)\}.
 $$
 If the set $\,\Gamma(u_0,u_1)
  := \set{\gamma\in C\big([0,1],\mathcal M\big)}{\gamma(0) = u_0,\ \gamma(1) =u_1}$ is not empty, then 
  $$ \alpha = \inf_{\gamma\in\Gamma(u_0,u_1)}
  \Big[\,\max_{w\in\gamma([0,1])} E(w) \,\Big] $$ 
  is a critical value for $\widetilde E$, i.e., there exists $u\in\mathcal M$ such that $E(u)=\alpha$ and $\|D\widetilde E(u)\|_{*}=0$.
\end{Thm}

\subsubsection{Eigenvalues of the $p$-Laplacian operator}\noindent
\\
Here we provide the definitions of eigenvalues and eigenfunctions of the $p$-Laplacian, along with the 
notion of simplicity of eigenvalues. Although these definitions are  quite similar to the ones  for open sets, some subtle aspects will be investigated later in this  section.
\begin{Def}\label{def:eigen} 
{\rm 
Let $A\subset\R^n$ be a $p$-quasi open set of finite measure and $\lambda\in\R$.  If
there exists a non-zero weak solution  of the eigenvalue problem 
 \begin{equation}\label{eq:weak eigen}  
 \begin{cases}
-\dv(|\gr u|^{p-2}\gr u)  =\ \lambda |u|^{p-2}u \ \ \text{in} \ A \\
 u \in W^{1,\,p}_0(A),
 \end{cases}
\end{equation}
then $\lambda$ is an {\it eigenvalue} of  the $p$-Laplacian on $A$ and $u$ is a corresponding {\it eigenfunction}; if in addition $\|u\|_{L^p(A)}=1$, then $u$ is called a {\it normalized} eigenfunction. The subspace of $W^{1,\,p}_0(A)$ generated by all the eigenfunctions corresponding to $\lambda$ is the  {\it eigenspace} associated to $\lambda$. When this eigenspace is  one-dimensional, i.e.
 $\{u,-u\}$ are the only normalized eigenfunctions, then $\lambda$ is said to be {\it simple}.
 }
\end{Def}
Note that \eqref{eq:weak eigen} is a special case of the Euler-Lagrange equation \eqref{eq:constrained problem} if one takes
 $X = W^{1,\,p}_0(A) $, 
\begin{equation}\label{casopart}
 E(w) =\int_A |\gr w|^p\dx\quad \text{and}\quad G(w)=\int_A |w|^p\dx. 
 \end{equation} 
Moreover, taking $u$ as test function in \eqref{eq:weak eigen}, one gets
\begin{equation}\label{eq:lambda}
 \lambda \ =\displaystyle\ \frac{\displaystyle\int_A |\gr u|^p \dx}{\displaystyle\int_A | u|^p \dx}.
\end{equation}
Hence, just as for open sets, it is easy to show that    
\begin{equation}\label{andr}
\lambda\, \geq\, c(n,p)\,|A|^{-p/n},
\end{equation} 
using the Sobolev inequality. Thus all the eigenvalues are bounded away from zero. 
\begin{Def}[(First eigenvalue)]\label{def:1 eigen}
{\rm 
 The first eigenvalue of the $p$-Laplacian in a $p$-quasi open set $A\subset\R^n$ of finite measure, is defined as 
$$
  \lambda_1(A) := \min\set{\lambda > 0}{\lambda \ \text{is an eigenvalue of the $p$-Laplacian in $A$}}.
$$
}
\end{Def}
Note that the above definition is well posed thanks to Proposition \ref{prop: first eigen prop} (i) below and to the fact that by  \eqref{andr} the eigenvalues are greater than a strictly positive constant. 
Introducing the manifold
\begin{equation}\label{manip}
\M_p(A):= \set{w\in W^{1,\,p}_0(A)}{\|w\|_{ L^p(A)} = 1},
\end{equation}
we have that 
\begin{equation}\label{fili1}
\lambda_1(A)= \inf_{w\in \M_p(A)} \int_A |\gr w|^p \dx.
\end{equation} 
Indeed, the fact that $\lambda_1(A)$ is greater than or equal to the infimum at the right hand side is an immediate consequence of \eqref{eq:lambda}. Instead, the opposite inequality
 is  
obtained by observing that if $A$ has finite measure then  the functional $E$ admits a minimizer $u$ on $\mathcal M_p(A)$ and $u$  is a weak solution of \eqref{eq:weak eigen} for some $\lambda>0$. Thus, $\lambda=\lambda_1(A)$ and  we have 
$$
\lambda_1(A) =\int_A |\gr u|^p \dx=\min_{w\in \M_p(A)}\ \int_A |\gr w|^p \dx. 
$$
In particular $u$ is an eigenfunction for $\lambda_1(A)$.

\begin{Rem}\label{marc}
{\rm 
For a  $p$-quasi open set $A$ of finite measure, set $X=W^{1,\,p}_0(A)$ and $E,G$ as in \eqref{casopart}. Note that for every $t\in\R$, $u\in W^{1,\,p}_0(A)$, the derivative $DE(u)-tDG(u)$ is the element of $W^{-1,\,p^\prime}(A)$ such that for all $\phi\in W^{1,\,p}_0(A)$
$$
\langle DE(u)-tDG(u),\phi\rangle=p\int_A(|\gr u|^{p-2}\gr u\cdot\gr\phi-t |u|^{p-2}u\phi)\,dx.
$$
From this equality it follows immediately that if $u$ is a critical point for $\widetilde E$ on $\mathcal M_p$, then $u$ is an eigenfunction and $E(u)$ is an eigenvalue of the $p$-Laplacian.
}
\end{Rem}
The following simple lemma will allow us to use Theorem~\ref{thm:Mpass}.
\begin{Lem}\label{utile} Let $A$ be a  $p$-quasi open set of finite measure. Let $E, G:W^{1,\,p}_0(A)\to\R$ be as in \eqref{casopart} and $\mathcal M_p$ as in \eqref{manip}. Then $E$ satisfies the Palais-Smale condition on $\mathcal M_p(A)$.
\end{Lem}
\begin{proof}
Let $u_h\in\mathcal M_p(A)$ be a sequence such that $E(u_h)$ is bounded and $\|D\tilde E(u_h)\|_*\to0$. Then there exists a sequence $t_h\in\R$ such that $\|\lap_pu_h+t_h|u_h|^{p-2}u_h\|_{W^{-1,\,p^\prime}\!(A)}\to0$. Setting $f_h:=\lap_pu_h+t_h|u_h|^{p-2}u_h$, we observe that
\begin{align*}
-\int_A|\gr u_h|^p\,dx+t_h=\langle f_h,u_h\rangle\to0\qquad\text{as $h\to\infty$.}
\end{align*}
Therefore, up to a subsequence we may assume that $t_h\to t\in\R$. Moreover, the compact imbedding of $W^{1,\,p}(B_r)$ in $L^p(B_r)$ for all $r>0$ and the Sobolev inequality imply that, up to a subsequence,  the sequence $u_h$ converges strongly in $L^p(A)$. In order to prove the lemma it is enough to show that $u_h$ converges strongly in $W^{1,\,p}_0(A)$. This follows immediately by observing  that  the sequence $\lap_pu_h$ converges strongly in $W^{-1,\,p^\prime}\!(A)$.
Therefore, we have
\begin{align*}
&\int_A(|\gr u_k|^{p-2}\gr u_k-|\gr u_h|^{p-2}\gr u_h)\cdot(\gr u_k-\gr u_h)\,dx=-\langle\lap_pu_k-\lap_pu_h, u_k-u_h\rangle  \\
&\qquad\qquad\qquad\qquad\qquad\qquad\qquad\qquad
\leq C\|\lap_pu_k-\lap_pu_h\|_{W^{-1,\,p^\prime}\!(A)}\to0 \quad\text{as $h,k\to\infty$},
\end{align*}
and the conclusion follows from Lemma~\ref{lem:ineq}.
\end{proof}
When $X$ is a Hilbert space (corresponding to $W^{1,\,2}_0(A)$ in this setting), then the 
discreteness of the spectrum is well known from the classical theory of linear operators. But for 
the case of $W^{1,\,p}_0(A)$ with $p\neq 2$, the existence of a spectral gap is not known in general, 
except for the first and second eigenvalues on open connected sets 
(see \cite{Lind1}, \cite{Juu-Lind}). 

To make matters worse, 
we work in the framework of quasi open sets with a weaker notion of  connectedness than the standard one. Therefore some delicate issues have to be handled  
in order  to characterize the second eigenvalue. 

 \subsection{Properties of the first eigenvalue}\label{subsec: Properties of fst}\noindent
For open sets of finite measure it is well known that every eigenvalue is the first eigenvalue 
in its nodal domains, i.e., if $\lambda$ is an eigenvalue with eigenfunction $u$, then 
$\lambda = \lambda_1(\{u>0\})$. The proof of this result for the eigenvalues of the $p$-Laplacian in an open set is due to 
Brasco-Franzina \cite[Th.~3.1]{Bra-Fra--hidconv}. The same proof carries on in the framework of quasi open sets, so we omit it. 
\begin{Lem}\label{lem:first eigen}
Let $A$ be a $p$-quasi open set of finite measure, $\lambda$ an eigenvalue of the  $p$-Laplacian in $A$ and  $u\in W_0^{1,\,p}(A)$ a corresponding eigenfunction. Then $\lambda=\lambda_1(\{u>0\})$ and $u$ is a first eigenfunction of $\{u>0\}$. Moreover, if $\lambda_1(A)$ is simple and 
$u\in W_0^{1,\,p}(A)$ is an eigenfunction of $\lambda_1(A)$, then $u$ does not change sign. 
\end{Lem}
From the above lemma and the minimum principle Theorem~\ref{thm:strong min},  we have the following result.
\begin{Cor}\label{cor:first eigen support}
Let $A$ be a $p$-quasi open set of finite measure such that $\lambda_1(A)$ is simple and let $u$ be a nonnegative first eigenfunction. Then 
$\{u>0\}$  is a $p$-quasi connected component of $A$ and $\lambda_1(A)=\lambda_1(\{u>0\})$.
\end{Cor}
\begin{proof} 
First, observe that by the minimum principle if $u$ is not identically zero in a $p$-quasi connected component $A'$ then it is strictly positive in $A'$.

We argue by contradiction assuming that there exist two different quasi connected components $A_1$ and $A_2$ of $A$ where $u$ is not identically zero and denote by $\tilde u_i$ the restriction of $u$ to $A_i$, for $i=1,2$. By Lemma~\ref{lem:quasi res}, we have $\tilde u_i\in W^{1,\,p}_0(A_i)$. Moreover, 
$$-\lap_p\tilde u_i=\lambda_1(A) |\tilde u_i|^{p-2}\tilde u_i, \qquad\text{on $A_i$}.$$
Hence, the  functions $\tilde u_1\pm\tilde u_2$ are two linearly independent eigenfunctions of $\lambda_1(A)$, which contradicts the assumption that $\lambda_1(A)$ is simple. 
\end{proof}
Using 
Lemma~\ref{lem:first eigen} and Corollary~\ref{cor:first eigen support},  we have the following 
proposition, which was proved  for open, connected sets in 
\cite{Lind1}, \cite{Lind2}. 
\begin{Prop}\label{prop: first eigen prop}
Let $A\subset \R^n$ be a $p$-quasi open set of finite measure. We have the 
following. 
\begin{enumerate}[(i)]
\item If  $\lambda_k$ is a sequence of eigenvalues such that $\lambda_k\to \lambda$, then 
$\lambda$ is also an eigenvalue. 
\item If the first eigenvalue $\lambda_1(A)$ is simple then it is isolated. 
\end{enumerate}
\end{Prop}

\begin{proof}
Statement (i) can be proved with the same   argument used in the proof  of \cite[Th.~3]{Lind2}. 

In order to prove (ii), we first consider the case when $A$ is $p$-quasi connected. Under this assumption, we claim that the  eigenfunctions corresponding to an eigenvalue $\lambda>\lambda_1(A)$ must change sign. 
To see this, assume that $u$ is a nonnegative eigenfunction.
By Theorem~\ref{thm:strong min}, $\{u>0\}$ coincides with $A$ up to a set of 
capacity zero. Then by Lemma \ref{lem:first eigen}  we conclude that 
$u$ is a first eigenfunction and the claim follows. 

To show that  the first eigenvalue 
is isolated we argue by contradiction, as in \cite[Th.~9]{Lind2}. Assume that there exists a sequence of eigenvalues $\lambda_k>\lambda_1(A)$ converging to $\lambda_1(A)$ and let 
 $u_k$ be the normalized eigenfunction corresponding to $\lambda_k$. We may assume that, up to a not relabelled subsequence, $ u_k$ converges weakly in $W^{1,\,p}_0(A)$, strongly in $L^p(A\cap B_r)$ for all $r>0$ and a.e. to a function $u\in W^{1,\,p}_0(A)$. Moreover, since $A$ has finite measure, a simple argument based on Sobolev inequality shows  (also for $p=n$) that $\|u\|_{L^p(A)}=1$. Hence, by lower semicontinuity, we get that 
$$ \int_A |\gr u|^p\dx \leq \lim_{k\to \infty}\lambda_k =\lambda_1(A). $$
Thus $u$ is a normalized  first eigenfunction. By Theorem~\ref{thm:strong min},  we may also assume without loss of generality  that $u>0$ q.e. on $A$. Now, 
arguing as in the proof of \eqref{andr}, one has
$$
\min\big\{ \,|\{u_k>0\}|\,,\, |\{u_k<0\}|\,\big\}\geq c(n,p)\lambda_k^{-\frac{n}{p}}$$ 
for all $k\in\N$.
Therefore, setting $A^+:=\limsup_{k\to\infty}\{u_k>0\}$, $A^-:=\limsup_{k\to\infty}\{u_k<0\}$, from the previous inequality, recalling that $A$ has finite measure, we have immediately that
$$
\min\{|A^+|,|A^-|\}\geq c(n,p)\lambda_1(A)^{-\frac{n}{p}} .$$ 
On the other hand, since the sequence $u_k$ is converging a.e. to $u$, we have also that $A^+\subset\{u\geq0\}$ and $A^-\subset\{u\leq0\}$ up to a set of zero Lebesgue measure. But the latter inclusion is impossible since $|A^-|>0$ and $u>0$ q.e. in $A$. This contradiction proves the result in this case.

Let us now assume that $A$ is not $p$-quasi connected and again let us argue by contradiction assuming that there exists a sequence of eigenvalues $\lambda_k>\lambda_1(A)$  converging to $\lambda_1(A)$. As before, we denote by $u_k$ a normalized eigenfunction of $\lambda_k$. Again, we may assume that, up to a not relabelled subsequence, $u_k$ converges strongly in $L^p_\loc(A)$ and a.e. in $A$ to a nonnegative normalized first eigenfunction $u$.
By Corollary~\ref{cor:first eigen support} the set $\{u>0\}$ is a $p$-quasi connected component of $A$. Thus by Lemma~\ref{lem:quasi res} we have that $u_k\in W^{1,\,p}_0(\{u>0\})$ for all $k$. Therefore each eigenfunction $u_k$ has to change sign in $\{u>0\}$, otherwise by Lemma~\ref{lem:first eigen} $u_k$ is a first eigenfunction in $\{u>0\}$ and $\lambda_k=\lambda_1(\{u>0\})=\lambda_1(A)$, which is impossible. Then the conclusion of the proof goes exactly as the preceeding case.
\end{proof}

The following proposition extends to $p$-quasi open and $p$-quasi connected sets a property that is well known in the case of open sets, see for instance \cite{AH} or \cite{Bel-Kaw}. 

\begin{Prop}\label{extAH}
Let  $A$ be a $p$-quasi open and $p$-quasi connected set of finite measure. Then $\lambda_1(A)$ is simple.
\end{Prop}
\begin{proof} First, observe that if $u$ is a first eigenfunction, then also $|u|$ is a first eigenfunction. Thus, by Theorem~\ref{thm:strong min} $u\not=0$ q.e. in $A$. Therefore, since $A$ is $p$-quasi connected,  $u$ does not change sign in $A$. Thus, in order to prove $\lambda_1(A)$ is simple, it is enough to show that if $u,v$ are two nonnegative normalized first eigenfunctions, then $u=v$. 

To this end, fix $\eps>0$ and recall the following extension of the classical Picone's identity, see \cite{AH}. For every two nonnegative functions $u,v\in W^{1,p}(\R^n)$ it holds true that
\begin{align}\label{extAH0}
0&
\leq|\gr u|^p+(p-1)\frac{u^p}{(v+\eps)^p}|\gr v|^p-p\frac{u^{p-1}}{(v+\eps)^{p-1}}|\gr v|^{p-2}\gr v\cdot\gr u \\
&=|\gr u|^p-|\gr v|^{p-2}\gr\bigg(\frac{u^p}{(v+\eps)^{p-1}}\bigg)\cdot\gr v. \nonumber
\end{align}
Integrating the right hand side of the previous inequality in $A$ and using the fact that $v$ is a first eigenfunction we  get
$$
\int_{A}|\gr u|^p\,dx-\int_A|\gr v|^{p-2}\gr v\cdot\gr\bigg(\frac{u^p}{(v+\eps)^{p-1}}\bigg)\,dx
=\int_{A}|\gr u|^p\,dx-\lambda_1(A)\int_A\frac{u^pv^{p-1}}{(v+\eps)^{p-1}}\,dx.
$$
Therefore, recalling \eqref{extAH0} and using Fatou's lemma we have, letting $\eps\to0$,
$$
\int_A\bigg(|\gr u|^p+(p-1)\frac{u^p}{v^p}|\gr v|^p-p\frac{u^{p-1}}{v^{p-1}}|\gr v|^{p-2}\gr v\cdot\gr u\bigg)\,dx=0.
$$
Recalling that by the minimum principle $v>0$ q.e. in $A$, a simple argument shows that the equality above implies  that for a.e. $x\in A$
$$
\nabla u(x)=\frac{u(x)}{v(x)}\nabla v(x).
$$
The conclusion then follows from Lemma~\ref{sob}, recalling that $\|u\|_{L^p}=\|v\|_{L^p}$.
\end{proof}

\subsection{Variational characterisation of the second eigenvalue}
The fact that if  $\lambda_1(A)$ is simple then it is isolated  shows that there is a spectral gap between 
$\lambda_1(A)$ and the next eigenvalue. This naturally leads to the following 
definition of second eigenvalue, which is well posed 
due to Proposition \ref{prop: first eigen prop}. 
\begin{Def}[(Second Eigenvalue)]\label{def:sec eigen}
{\rm 
Let $A\subset \R^n$ be a $p$-quasi open set of finite measure. The second eigenvalue of the $p$-Laplacian on $A$ is defined as follows.
 \begin{equation}\label{eq:sec eigen}
  \lambda_2(A):= 
  \begin{cases}
   \ \min\set{\lambda > \lambda_1(A)}{\lambda \ \text{is an eigenvalue}}
   \qquad\quad &\text{if}\ \lambda_1(A)\ \text{is simple}\\
  \  \lambda_1(A) \qquad &\text{otherwise}
  \end{cases}
 \end{equation}
 }
\end{Def}
Equipped with \eqref{eq:sec eigen}, now we restate Theorem \ref{thm:mainthm1}, which is a {\it mountain pass characterisation} of the second eigenvalue. This characterisation is the main result of this paper. 
 \begin{Thm}\label{prop:mount pass}
Let $A$ be a $p$-quasi open set of finite measure, $u_1 \in W^{1,\,p}_0(A)$ be a normalized 
 eigenfunction of $\lambda_1(A)$,  and 
 $$\Gamma(u_1,-u_1)
 =\set{\gamma\in C\big([0,1],\M_p(A)\big)}{\gamma(0) = u_1,\ \gamma(1) = -u_1},$$
 where $\mathcal M_p(A)$ is the manifold defined in \eqref{manip}. Then
 \begin{equation}\label{eq:mount pass}
  \lambda_2(A) =\min_{\gamma\in\Gamma(u_1,-u_1)}
  \bigg[\,\max_{w\in\gamma([0,1])} \int_A |\gr w|^p \dx\,\bigg]. 
 \end{equation}
\end{Thm}
The  result above was proved in the case of open sets  in \cite{Bra-Fra--hks}. 
However, the techniques used therein cannot be adopted in the 
setting of quasi open sets. It is noteworthy that when $\lambda_1(A)$ is simple the  characterisation 
in \eqref{eq:mount pass} coincides with the one given using the Krasnoselskii genus. This latter definition was used by  Anane-Tsouli \cite{An-Ts} to prove the existence of spectral gap between the first and the second eigenvalue. We refer the reader to 
\cite{Lind1}, \cite{Struwe}, \cite{Juu-Lind}  for further details. 

\subsubsection{A path of decreasing $p$-energy}\noindent
\\ 
To prove Theorem \ref{prop:mount pass}, we
construct an appropriate path, so that the $p$-Dirichlet energy decreases throughout the 
path and remains lower than $\lambda_2(A)$. The idea is similar to the one used in \cite{Cue-Fig-Gos}, but the construction of the path in our case is completely different.  Indeed, our construction relies on the theory of minimizing movements, a  technique introduced by De Giorgi and developed in the book by 
Ambrosio-Gigli-Savar\'e \cite{Amb-Gig-Sav}. To this end, we need to introduce the basic notation and definitions. For further details we refer to the first two chapters of \cite{Amb-Gig-Sav}.

Let $(\mathcal S,d)$ be a complete metric space and  $\Phi:\mathcal S\to(-\infty,+\infty]$. Denote by 
\begin{equation}\label{effe}
\mathscr D(\Phi):=\set{v\in \mathcal S}{\Phi(v)<\infty}
\end{equation}
 the {\it effective domain} of $\Phi$, i.e., the set of points where $\Phi$ is finite. Given a point $v\in \mathscr D(\Phi)$, the  {\it local slope of $\Phi$ at $v$} is defined by setting
\begin{equation}\label{locslo}
|\partial\Phi|(v):= \limsup_{w\to v}\frac{(\Phi(v)-\Phi(w))^+}{d(v,w)}\ .
\end{equation}

In order to apply the results of \cite{Amb-Gig-Sav} we now specify our choice for $(\mathcal S,d)$ and $\Phi$. More precisely, in this section we are going to take as a metric space 
the space $L^p(A)$ with the distance induced by the norm, where $A$ is a quasi open set of finite measure. The functional $\Phi$ will be defined as follows
\begin{equation}\label{AGS1}
\Phi(v):= 
\begin{cases}
\displaystyle E(v) & \text{if $v\in \mathcal M_p(A)$} \cr
+\infty & \text{if $v\in L^p(A)\setminus\mathcal M_p(A)$}, \cr
\end{cases}
\end{equation}
where $E$ is the functional defined  in \eqref{casopart} and $\mathcal M_p(A)$ is given in \eqref{manip}.
Next lemma is a key ingredient in the proof of the existence of a path having all the properties stated in Lemma~\ref{lem:path}.
\begin{Lem}\label{tech}
Let $A \subset \R^n$ be a $p$-quasi open set of finite measure and let $\Phi$ be the functional defined in \eqref{AGS1}. For all $v\in\mathcal M_p(A)$, we have  
\begin{equation}\label{tech1}
|\partial\Phi|(v)\geq \frac{p}{2}\|\lap_pv+E(v)|v|^{p-2}v\|_{W^{-1,\,p^\prime}\!(A)}.
\end{equation}
Moreover, if $v_0\in\mathcal M_p(A)$ is not an eigenfunction for the $p$-Laplacian,  there exist  $c_0,\delta>0$ such  that 
\begin{equation}\label{tech2}
\|\lap_pv+E(v)|v|^{p-2}v\|_{W^{-1,\,p^\prime}\!(A)}\geq c_0\quad
\text{for all}\ v\in\mathcal M_p(A) \ \ \text{satisfying}\ \ \|v-v_0\|_{L^p(A)}<\delta.
\end{equation}
\end{Lem}
\begin{proof}
Fix $v\in\mathcal M_p(A)$ and  $\varphi\in W^{1,\,p}_0(A)$, not parallel to $v$. Setting $w_t=(v+t\varphi)/\|v+t\varphi\|_{L^p(A)}$, a simple calculation shows that
$$
\lim_{t\to0^+}\frac{\Phi(w_t)-\Phi(v)}{\|w_t-v\|_{L^p(A)}}=p\frac{\langle-\lap_pv-E(v)|v|^{p-2}v\ ,\ \varphi\rangle}{\|\langle|v|^{p-2}v,\varphi\rangle v-\varphi\|_{L^p(A)}},
$$
where, as usual, we denote 
$$
\langle-\lap_pv,\varphi\rangle=\int_A|\gr v|^{p-2}\gr v\cdot\gr\varphi\,dx
\quad\text{and}\quad\langle|v|^{p-2}v,\varphi\rangle=\int_A|v|^{p-2}v\varphi\,dx.
$$
Thus, observing that $\|\langle|v|^{p-2}v,\varphi\rangle v-\varphi\|_{L^p(A)}\leq2\|\varphi\|_{L^p(A)}$, we have
\begin{align*}
|\partial\Phi|(v)
&\geq p\sup\bigset{\frac{|\langle\lap_pv+E(v)|v|^{p-2}v,\varphi \rangle|}{\|\langle|v|^{p-2}v,\varphi\rangle v-\varphi\|_{L^p(A)}}}{\varphi\in W^{1,\,p}_0(A), \, \varphi\not=tv \ \text{for $t\in\R$}} \\
&\geq p\sup\bigset{\frac{|\langle\lap_pv+E(v)|v|^{p-2}v,\varphi \rangle|}{2\|\varphi\|_{W^{1,\,p}_0(A)}}}{\varphi\in W^{1,\,p}_0(A)} \\
& =\ \frac{p}{2}\|\lap_pv+E(v)|v|^{p-2}v\|_{W^{-1,\,p^\prime}(A)}.
\end{align*}
This proves \eqref{tech1}. 

In order to prove \eqref{tech2} we argue by contradiction assuming that  \eqref{tech2} does not hold. If so,  there exists a sequence  $v_h\in\mathcal M_p(A)$ converging to $v_0$ in $L^p(A)$ and such that  
$$\lap_p v_h+E(v_h)|v_h|^{p-2}v_h\to0\ \ \text{in}\ W^{-1,\,p^\prime}(A).$$ 
Observe that the sequence $v_h$ is bounded in $W^{1,\,p}_0(A)$. In fact, if for a not relabelled subsequence $\|\gr v_h\|_{L^p(A)}\to+\infty$, then we have
$$
\int_A|v_h|^{p-2}v_hv_0\,dx=\frac{1}{E(v_h)}\int_A|\gr v_h|^{p-2}\gr v_h\cdot\gr v_0\,dx+\frac{1}{E(v_h)}\langle\lap_p v_h+E(v_h)|v_h|^{p-2}v_h,v_0\rangle.
$$
Thus, by H\"older inequality, we would have
$$
 \int_A|v_h|^{p-2}v_hv_0\,dx\leq \frac{\|\gr v_0\|_{L^p(A)}}{\|\gr v_h\|_{L^p(A)}}+\eps_h\|v_0\|_{W^{1,\,p}_0(A)},
 $$
where $\eps_h\to0$ as $h\to+\infty$, from which we would conclude that $v_0=0$, which is impossible since $v_0\in\mathcal M_p(A)$. Thus the sequence $E(v_h)$ is bounded and, up to a subsequence, we may assume that it converges to some number $E_0\geq E(v_0)>0$.  Moreover, since $\lap_p v_h+E(v_h)|v_h|^{p-2}v_h\to0$ in $W^{-1,\,p^\prime}(A)$ and $v_h\to v_0$ in $L^p(A)$,  also $\lap_p v_h$ converges strongly in $W^{-1,\,p^\prime}(A)$. Thus, arguing exactly as in the final part of the proof of Lemma~\ref{utile}, we conclude that $v_h\to v_0$ in $W^{1,\,p}_0(A)$.  Note that from the strong convergence of $v_h$ to $v_0$ in $W^{1,\,p}_0(A)$ and of $\lap_pv_h$ in $W^{-1,\,p^\prime}(A)$, we have that indeed $\lap_pv_h\to\lap_pv_0$ in $W^{-1,\,p^\prime}(A)$. Thus we get that 
$$\lap_pv_0+E_0|v_0|^{p-2}v_0=0,$$ which is impossible, since $v_0$ is not an eigenfunction. This contradiction concludes the proof of \eqref{tech2} and hence 
the proof of 
the lemma.
 \end{proof}
 We are now ready to give the proof of the next crucial lemma, which provides the construction of a low energy path connecting the first eigenfunction of the $p$-Laplacian to a function which is not an eigenfunction. With this lemma in hand, the proof of Theorem~\ref{prop:mount pass} will follow quickly.
\begin{Lem}\label{lem:path}
Let $A\subset \mathbb R^n$ be a $p$-quasi open set of finite measure. 
Suppose that $\lambda_1(A) $ is simple and let $u_1$ be the first nonnegative normalized eigenfunction. If $v_0 \in \M_p(A)$ is not an eigenfunction 
and         
 $ \lambda_1(A) < E(v_0) \leq \lambda_2(A) $, then there exists 
 a curve $ v\in C^{0,(p-1)/p}\big([0,\infty),\M_p(A)\big)\cap W^{1,\,p}\big([0,\infty),L^p(A)\big) $ with $v(0)= v_0$, such that the 
 following hold:
 \begin{align}
\label{1} 
  (i)&\quad E(v(t))< \lambda_2(A) \quad \forall\ t>0, \quad\text{and}\quad \int_0^{\infty}\|v'(t)\|^p_{L^p(A)}\,dt\,\leq E(v_0);\\
 \label{2} (ii)&\quad \lim_{t\to \infty} E(v(t)) =\lambda_1(A);\\
 \label{3}(iii)&\quad  \lim_{t\to \infty} v(t)=u_1 \quad \text{or} \quad \lim_{t\to \infty} v(t)=-   u_1 \ \ \text{in} \ \ W^{1,\,p}_0(A).
 \end{align}
\end{Lem}
\begin{proof}
{\it Step 1 (The discrete scheme).}
Fix  $\tau >0$ and set $ v^\tau_0 := v_0$. Then, for all $k\geq1$ we define recursively the function  $v^\tau_k$ by selecting a  minimizer of the following problem:
\begin{equation}\label{eq:minscheme}
\min_{w\in \M_p(A)}\bigg[\ \frac{1}{\tau^{p-1}}\int_A |w-v^\tau_{k-1}|^p\dx
+  \int_A |\gr w|^p\dx\ \bigg].
\end{equation}
The existence of a minimizer follows from coercivity and weak lower semicontinuity. Moreover, there exists a Lagrange multiplier $\sigma^\tau_k \in \R$ such that for every $\phi\in W^{1,\,p}_0(A)$ we have 
\begin{equation}\label{eq:weakscheme}
\frac{1}{\tau^{p-1}}\int_A |v^\tau_k-v^\tau_{k-1}|^{p-2}(v^\tau_k-v^\tau_{k-1})\phi\ dx
+  \int_A |\gr v^\tau_k|^{p-2}\gr v^\tau_k\cdot\gr\phi \ dx 
=\sigma^\tau_k \int_A |v^\tau_k|^{p-2}v^\tau_k\,\phi\ dx.
\end{equation}
Since $v^\tau_k \in \M_p(A)$, choosing $\phi = v^\tau_k$, we have 
\begin{equation}\label{eq:sigma k}
 \sigma^\tau_k=\int_A |\gr v^\tau_k|^p\dx
 + \frac{1}{\tau^{p-1}}\int_A |v^\tau_k-v^\tau_{k-1}|^{p-2}(v^\tau_k-v^\tau_{k-1})v^\tau_k\dx
\end{equation}
 Then, comparing  the values of the functional in \eqref{eq:minscheme} at 
$v^\tau_k$ and $v^\tau_{k-1}$  we get for all $k\geq1$
\begin{equation}\label{eq:Ediff}
\frac{1}{\tau^{p-1}}\|v^\tau_k -v^\tau_{k-1}\|^p_{L^p(A)}
\leq E(v^\tau_{k-1})- E(v^\tau_k).
\end{equation}
Now we choose an uniform partition of 
$[0,\infty) $ with $\{0,\tau, 2\tau,\ldots\}$ and define the piecewise 
constant flow $ v^\tau :[0,\infty) \to\ \M_p(A)$ by setting $ v^\tau(t)(x)\ = v^\tau_{[t/\tau]}(x)$ for all $t>0,\,x\in\R^n$, where 
$[\cdot]$ denotes the integer part function. Similarly, we denote by $\sigma^\tau$ the piecewise constant function from $[0,\infty)$ to $\R$ defined by setting $ \sigma^\tau(t)= \sigma^\tau_{[t/\tau]}$. Using \eqref{eq:Ediff}, we have that 
for all $ t> s\geq0$ with $[t/\tau]>[s/\tau]$
\begin{align*}
  \|v^\tau(t) -v^\tau(s) \|_{L^p(A)} \ &\leq \sum_{k= [s/\tau]+1}^{[t/\tau]} 
\|v^\tau_k -v^\tau_{k-1}\|_{L^p(A)}\\
&\leq \ \Big([t/\tau]-[s/\tau]\Big)^\frac{p-1}{p} 
\bigg(\sum_{k= [s/\tau]+1}^{[t/\tau]} \|v^\tau_k -v^
\tau_{k-1}\|^p_{L^p(A)}\bigg)^\frac{1}{p} \nonumber \\
&\leq \  \Big([t/\tau]-[s/\tau]\Big)^\frac{p-1}{p}
\tau^\frac{p-1}{p}\Big[ E(v^\tau_{[s/\tau]})- E(v^\tau_{[t/\tau]})\Big]^\frac{1}{p}. \nonumber
\end{align*}
Thus, we find that 
$$
 \|v^\tau(t) -v^\tau(s) \|_{L^p(A)}\leq  (E(v_0))^\frac{1}{p}( t-s +\tau)^\frac{p-1}{p}.  
$$
 From this inequality, recalling that the functions $\{ v^\tau\}_{0<\tau<1}$ are  bounded in $W^{1,\,p}_0(A)$ uniformly with respect to $t$, we deduce, thanks to a refined version of  Arzel\`a-Ascoli theorem (see \cite[Prop. 3.3.1]{Amb-Gig-Sav}), that there exists a sequence $\tau_i\to0$ such that for all $t>0$ the curves $v^{\tau_i}(t)$ converge in $L^p(A)$, uniformly with respect to $t\in[0,T]$, to a curve  $v \in C^{\,0\,,(p-1)/p}([0,\infty),L^p(A))$. Moreover, since for every $t>0$ the sequence $v^{\tau_i}(t)$ is bounded in $W^{1,p}_0(A)$, a simple compactness argument shows that it converges weakly in $W^{1,p}_0(A)$ to $v(t)$.
Thus $v(t)\in \M_p(A)$ for all $t\geq 0$. \\

{\it Step 2 (Convergence of the discrete scheme).}
We set
 $$
 \hat v_i(t):= \frac{(t-(k-1)\tau_i)v^{\tau_i}_k+(k\tau_i-t)v^{\tau_i}_{k-1}}{\tau_i}\qquad\text{for $t\in[(k-1)\tau_i,k\tau_i]$}
 $$ 
and observe that, up to another not relabelled subsequence,  the functions $\hat v_i'$ converge  weakly in $L^p([0,\infty),L^p(A))$ to $v'$. Indeed, this follows immediately from \eqref{eq:Ediff} since for every $T>0$
\begin{equation}\label{tech3}
\int_0^T\|\hat v_i'(t)\|_{L^p(A)}^p\,dt\,\leq \sum_{k=1}^{[T/\tau_i]+1}(E(v^{\tau_i}_{k-1})- E(v^{\tau_i}_k))\leq E(v_0).
\end{equation}
Note that from this inequality we get in particular the second estimate in \eqref{1}.
Next, observe that, again up to a not relabelled subsequence, we may assume that the functions $\sigma^{\tau_i}$ converge weakly in $L^{p^\prime}(0,T)$ for all $T>0$. In fact, from \eqref{eq:sigma k} we have, using H\"older inequality and recalling \eqref{tech3}, 
\begin{align*}
\int_0^T|\sigma^{\tau_i}(t)|^{p^\prime}\,dt\,&
\leq C\,TE(v_0)^{p^\prime} +C \int_0^T\bigg(\int_A|\hat v_i'(t)|^{p-1}|v^{\tau_i}(t)|\,dx\bigg)^{\frac{p}{p-1}}dt \\
& \leq C\,TE(v_0)^{p^\prime}+ C\int_0^T\|\hat v_i'(t)\|_{L^p(A)}^p\,dt\leq C\big(TE(v_0)^{p^\prime}+E(v_0)),
\end{align*}
for a suitable constant $C$ depending only on $p$.
Finally, we claim that $v^{\tau_i}(t)\to v(t)$ strongly in $W^{1,\,p}_0(A)$ for a.e. $t>0$. To prove this last claim we are going to use \eqref{eq:weakscheme} and the convexity of the functional $E(v)$. Precisely, we have that for a.e. $t>0$,
\begin{align*}
\int_A|\gr v(t)|^p\,dx
\ &\geq \int_A|\gr v^{\tau_i}(t)|^p\,dx+p\int_A|\gr v^{\tau_i}(t)|^{p-2}\gr v^{\tau_i}(t)\cdot(\gr v(t)-\gr v^{\tau_i}(t))\,dx \\
& =\ \int_A|\gr v^{\tau_i}(t)|^p\,dx+p\sigma^{\tau_i}(t)\int_A|v^{\tau_i}(t)|^{p-2}v^{\tau_i}(t)(v(t)-v^{\tau_i}(t))\,dx \\
&\qquad\qquad -p\int_A|\hat v_i'(t)|^{p-2}\hat v_i'(t)(v(t)-v^{\tau_i}(t))\,dx.
\end{align*}
Integrating the above inequality with respect to time, with some easy calculations we get that for every $T>0$
\begin{align*}
\int_0^T\int_A|\gr v(t)|^p\,dxdt 
 \ &\geq \int_0^T\int_A|\gr v^{\tau_i}(t)|^p\,dxdt- p\int_0^T|\sigma^{\tau_i}(t)|\bigg(\int_A|v(t)-v^{\tau_i}(t)|^p\,dx \bigg)^{\frac1p}\,dt\\
& \qquad\qquad -p\int_0^T\int_A|\hat v_i'(t)|^{p-2}\hat v_i'(t)(v(t)-v^{\tau_i}(t))\,dxdt.
\end{align*}
Therefore, recalling that the $\sigma^{\tau_i}$ are bounded in $L^{p^\prime}(0,T)$, that $v^{\tau_i}$ converge to $v$ in $L^p$ locally uniformly with respect to $t$  and that the $\hat v_i'$ are bounded in $L^p([0,\infty),L^p(A))$, passing to the limit we immediately get
\begin{align*}
\int_0^T\int_A|\gr v(t)|^p\,dxdt 
\ &\geq \liminf_{i\to\infty}\int_0^T\int_A|\gr v^{\tau_i}(t)|^p\,dxdt \\
\ &\geq  \int_0^T\bigg(\liminf_{i\to\infty}\int_A|\gr v^{\tau_i}(t)|^p\,dx\bigg)dt\ \geq  \int_0^T\int_A|\gr v(t)|^p\,dxdt,
\end{align*}
where we used Fatou lemma and the lower semicontinuity of the energy $E$ with respect to the weak convergence in $W^{1,\,p}_0$. Thus we have proved that for a.e. $t>0$
\begin{equation}\label{tech4}
\int_A|\gr v(t)|^p\,dx
=\liminf_{i\to\infty}\int_A|\gr v^{\tau_i}(t)|^p\,dx.
\end{equation}
Note that by \eqref{eq:Ediff} for every $i$ the fuctions $t\mapsto E(v^{\tau_i}(t))$ are decreasing. Therefore, by Helly's lemma (see \cite[Lemma 3.3.3]{Amb-Gig-Sav}), there exists a not relabelled subsequence such that for every $t>0$ there exists the limit of $E(v^{\tau_i}(t))$. This shows that, up to a subsequence, the $\liminf$ in \eqref{tech4} is indeed a limit, hence $\|\nabla v^{\tau_i}(t)\|_{L^p(A)}\to\|\nabla v(t)\|_{L^p(A)}$.  This, together with the weak convergence in $W^{1,p}_0(A)$ of $v^{\tau_i}(t)$   to $v(t)$ proved in Step 1, implies that, up to a not relabelled subsequence,  $v^{\tau_i}(t)$ converges to $v(t)$ in $W^{1,\,p}_0(A)$ for a.e. $t>0$.\\

{\it Step 3 (An energy inequality).} We claim that there exists $c(p)>0$ such that for a.e. $t>0$
\begin{equation}\label{eq:en id}
\int_0^{t}\|v'(s)\|^p_{L^p(A)}\,ds+c(p)\int_0^{t}\|\lap_pv(s)+E(v(s))|v(s)|^{p-2}v(s)\|_{W^{-1,\,p^\prime}(A)}^{p^\prime}\,ds\leq E(v_0)-E(v(t)).
\end{equation}
To this end, we introduce a third kind of interpolation due to De Giorgi. For every $t\in(\tau_i(k-1),\tau_ik]$, $k\geq1$, we denote by $\tilde v_i(t)$ a minimizer of
$$
\min_{w\in \M_p(A)}\bigg[\ \frac{1}{(t-\tau_i(k-1))^{p-1}}\int_A |w-v^{\tau_i}_{k-1}|^p\dx
+  \int_A |\gr w|^p\dx\ \bigg].
$$
Just as in \eqref{eq:Ediff}, here we have that for every $t\in(\tau_i(k-1),\tau_ik]$
\begin{equation}\label{tech5}
\frac{1}{(t-\tau_i(k-1))^{p-1}}\|\tilde v_i(t) -v^{\tau_i}_{k-1}\|^p_{L^p(A)}
\leq  E(v^{\tau_i}_{k-1})- E(\tilde v_i(t)).
\end{equation}
Hence, we have that for every $t>0$
$$
\|\tilde v_i(t) -v^{\tau_i}(t)\|_{L^p(A)}\leq  E(v_0)^{\frac1p}\tau_i^{\frac{p-1}{p}}.
$$
and from this inequality we conclude at once that for all $T>0$ also the curves $\tilde v_i(t)$ converge strongly in $L^p(A)$ to $v(t)$ uniformly with respect to $t\in[0,T]$. Note also that from \eqref{tech5}, for every $t>0$ we have $E(\tilde v_i(t))\leq E(v^{\tau_i}(t))$. Thus, for a.e. $t>0$
\begin{align*}
E(v(t)) \leq \liminf_{i\to\infty}E(\tilde v_i(t))\leq \limsup_{i\to\infty}E(\tilde v_i(t))\leq \lim_{i\to\infty}E(v^{\tau_i}_k(t))=E(v(t)).
\end{align*}
Therefore we may conclude that also the functions $\tilde v_i(t)$ converge strongly in $W^{1,\,p}_0(A)$ to $v(t)$ for a.e. $t>0$.
Now a very general argument which uses only the definition and no special properties of the local slope defined in \eqref{locslo} shows that for the interpolation defined above one has for every $i$ and every $k\geq1$
$$
\int_0^{k\tau_i}\|\hat v_i'(s)\|^p_{L^p(A)}\,ds +\frac{(p-1)}{p^{p^\prime}}\int_0^{k\tau_i}(|\partial\Phi|(\tilde v_i(s)))^{p^\prime}\,ds\leq E(v_0)-E(v^{\tau_i}_k),
$$
see the inequalities (3.2.16) and (3.2.17) in \cite{Amb-Gig-Sav}, where $\Phi$ is defined as in \eqref{AGS1}.
Thus, recalling \eqref{tech1} we deduce that for all $i$ and  for all $t>0$,  setting $c(p):=(p-1)/2^{p^\prime}$,  we have
$$
\int_0^{t}\|\hat v_i'(s)\|^p_{L^p(A)}ds+c(p)\!\int_0^{t}\|\lap_p\tilde v_i(s)+E(\tilde v_i(s))|\tilde v_i(s)|^{p-2}\tilde v_i(s)\|_{W^{-1,\,p^\prime}(A)}^{p^\prime}ds\leq E(v_0)-E(v^{\tau_i}(t)).
$$
Recalling that $\hat v_i^\prime$ converges weakly in $L^p([0,\infty),L^p(A))$ to $v'$ and that $\tilde v_i(t)$ and  $v^{\tau_i}(t)$ converge in $W^{1,\,p}_0(A)$ to $v(t)$ for a.e. $t>0$,  \eqref{eq:en id} follows letting $i\to\infty$.\\

{\it Step 4 (Conclusion of the proof).}
By \eqref{eq:Ediff}, for every $i$, the function $t\mapsto E(v^{\tau_i}(t))$ is decreasing. Therefore, denoting by $Z_0\subset(0,\infty)$ a set of zero  $\mathcal L^1$ measure such  that $v^{\tau_i}(t)$ converges to  $v(t)$ in $W^{1,\,p}_0(A)$ for all $t\in(0,\infty)\setminus Z_0$, we have \begin{equation}\label{bon1}
\qquad\qquad\qquad\qquad E(v(t))\leq E(v(s))\quad\text{ for all $0<s<t$ with $s,t\not\in Z_0$}.
\end{equation}
Since $v \in C^{\,0\,,(p-1)/p}([0,\infty),L^p(A))$, if $s\to t$ then $\|v(s)-v(t)\|_{L^p(A)}\to0$, hence by lower semicontinuity we have also
\begin{equation}\label{bon2}
E(v(t))\leq\liminf_{s\to t}E(v(s)) \quad \text{for all $t>0$.}
\end{equation}
 Moreover, since by assumption $v_0$ is not an eigenfunction, from Lemma~\ref{tech} and from the fact that $v \in C^{\,0\,,(p-1)/p}([0,\infty),L^p(A))$, it follows that there exist $t_0, c_0>0$ such that 
$$\|\lap_pv(t)+E(v(t))|v(t)|^{p-2}v(t)\|_{W^{-1,\,p^\prime}(A)}\geq c_0 \ \ 
\text{for all}\ t\in[0,t_0].$$ 
Hence, \eqref{eq:en id}, \eqref{bon1} and \eqref{bon2} yield that $E(v(t))<E(v_0)\leq\lambda_2(A)$ for all $t>0$. This proves the first inequality in \eqref{1}. Note that the assumption that  $v_0$ is not an eigenfunction  is crucial for the validity of  such estimate. Indeed, if $v_0$ were an eigenfunction then the above construction would produce the limit flow $v(t)\equiv v_0$ for all $t>0$.

Now, let us set 
\begin{equation}\label{tech10}
\alpha:= \lim_{\substack{t\to+\infty \\ t\not\in Z_0}} E(v(t))< \lambda_2(A).
\end{equation}
This limit exists and it is strictly smaller than $\lambda_2(A)$ since the function $E(v(t))$ is decreasing for $t\not\in Z_0$ and $E(v(t))<\lambda_2(A)$ for all $t>0$. 


Note that \eqref{tech3} yields
$$
\int_0^{\infty}\||\hat v_i'(t)|^{p-2}\hat v_i'(t)\|_{L^{p^\prime}\!(A)}^{p^\prime}\,dt\,\leq  E(v_0),
$$
for every $i$. 
Therefore,  we may assume that, up to a not relabelled subsequence,   $|\hat v_i'|^{p-2}\hat v_i'$ converges weakly in $L^{p^\prime}([0,\infty),L^{p^\prime}(A))$  to a curve $q$ such that
\begin{equation}\label{tech7}
\int_0^{\infty}\|q(t)\|_{L^{p^\prime}\!(A)}^{p^\prime}\,dt\,\leq  E(v_0).
\end{equation}
Now, let us integrate \eqref{eq:weakscheme} in $(0,t)$ and let us pass to the limit as $i\to\infty$. From the weak convergence of $|\hat v_i'|^{p-2}\hat v_i'$ in $L^{p^\prime}([0,\infty),L^{p^\prime}(A))$ and all the convergences proved in Step 2 we have, that for all $t>0$
$$
\int_{0}^{t}\int_A q(t)\phi\ dxdt
+ \int_{0}^{t} \int_A |\gr v(t)|^{p-2}\gr v(t)\cdot\gr\phi \ dxdt
 =\int_{0}^{t}\sigma(t)\,dt \int_A |v(t)|^{p-2}v(t)\phi\ dx.
$$
for all $\phi\in\mathcal D$, where $\mathcal D$ is a dense sequence in $W^{1,\,p}_0(A)$. Differentiating this equality with respect to $t$ yields that for a.e. $t>0$ and for all $\phi\in\mathcal D$
\begin{equation}\label{tech8}
\int_A q(t)\phi\ dx
+ \int_A |\gr v(t)|^{p-2}\gr v(t)\cdot\gr\phi \ dx
= \sigma(t) \int_A |v(t)|^{p-2}v(t)\phi\ dx.
\end{equation}
By density, this equation holds for a.e. $t>0$ and for every $\phi\in W^{1,\,p}_0(A)$. Now, let us choose a sequence $t_h\in(0,\infty)\setminus Z_0$ such that \eqref{tech8} holds, $t_h\to+\infty$,  $\|q(t_h)\|_{L^{p^\prime}\!(A)}^{p^\prime}\to0$ as $h\to\infty$. Note that this is possible thanks to \eqref{tech7}. Observe that since the sequence $v(t_h)$ is bounded in $W^{1,p}_0(A)$ up to a subsequence it converges strongly in $L^p(A)$ and weakly in $W^{1,p}_0(A)$ to a function $w\in\M_p(A)$. Testing the equation satisfied by $v(t_h)$ with $v(t_h)$,  we have that
\begin{equation}\label{tech9}
\lim_{h\to\infty}\sigma(t_h)=\lim_{h\to\infty}\bigg[E(v(t_h))+ \int_Aq(t_h)v(t_h)\,dx\bigg]=\alpha.
\end{equation}
Let us now fix $k>h\geq1$ and let us test with $v(t_k)-v(t_h)$ the equation \eqref{tech8} satisfied by $v(t_k)$ and the equation satisfied by $v(t_h)$. Subtracting the two resulting equations we have
\begin{align*}
&\int_A \big(|\gr v(t_k)|^{p-2}\gr v(t_k)-|\gr v(t_h)|^{p-2}\gr v(t_h)\big)\cdot(\gr v(t_k)-\gr v(t_h))\,dx \\
&\qquad\qquad\qquad =\int_A(\sigma(t_k)|v(t_k)|^{p-2}v(t_k)-\sigma(t_h)|v(t_h)|^{p-2}v(t_h))(v(t_k)-v(t_h))\,dx\\
&\qquad\qquad\qquad \qquad\qquad-\int_A(q(t_k)-q(t_h))(v(t_k)-v(t_h))\,dx.
\end{align*}
From this equation, recalling Lemma~\ref{lem:ineq} and \eqref{tech9}, and using  the  convergence of $v(t_h)$ to $w$ in $L^p$ and the convergence of $q(t_h)$ to $0$ in $L^{p^\prime}(A)$ we get that the sequence $\gr v(t_h)$ converges to $\gr w$ in $L^p(A)$.  Now, considering \eqref{tech8} at the time $t_h$ and letting $h\to\infty$, from \eqref{tech9} we finally get that for all $\phi\in W^{1,\,p}_0(A)$
$$
\int_A |\gr w|^{p-2}\gr w\cdot\gr\phi \ dx
 =\alpha \int_A |w|^{p-2}w\phi\ dx.
$$
Therefore $w$ is an eigenfunction. Then, from \eqref{tech10} we deduce that $\alpha=\lambda_1(A)$ and that $w$ is either equal to $u_1$ or to $-u_1$. 

Finally, observe that   $E(v(t))\geq\lambda_1(A)$ for all $t>0$. From this inequality, \eqref{bon2} and the fact that $E(v(t))\to\lambda_1(A)$ as $t\to+\infty$, $t\not\in Z_0$, we conclude that
\begin{equation}\label{finedim}
\lim_{t\to+\infty}E(v(t))=\lambda_1(A).
\end{equation}
 This establishes \eqref{2}. 
 
 To conclude the proof we need to show that $v(t)\to w$ in $L^p(A)$ as $t\to+\infty$. To this end we argue by contradiction assuming that there exists a sequence $s_h$, with $s_h\to+\infty$, such that $\|v(s_h)-w\|_{L^p(A)}\geq c>0$ for all $h$. Then, since the sequence $v(s_h)$ is bounded in $W^{1,p}_0(A)$, we may assume that, up to a not relabelled subsequence, it converges strongly in $L^p(A)$ and weakly in $W^{1,p}_0(A)$ to some function $z\in\M_p(A),  z\not=w$ and that $0<s_h<t_h$ for all $h$. Then, from \eqref{finedim} it follows that $E(z)=\lambda_1(A)$. This means that $z$ is a normalized eigenvalue and thus, since $\lambda_1(A)$ is simple, $z=-w$.  In particular, for $h$ sufficiently large we have
 $$
 \|v(t_h)-w\|_{L^p(A)}<\frac12,\qquad \|v(s_h)+w\|_{L^p(A)}<\frac12.
 $$
Note that, since $v\in C([0,\infty),L^p(A))$, the function $f(t)=\|v(t)-w\|_{L^p(A)}-\|v(t)+w\|_{L^p(A)}$ is continuous. From the two inequalities above it follows that 
$$
f(s_h)\geq 2\|w\|_{L^p(A)}-2\|v(s_h)+w\|_{L^p(A)}>1.
$$
Similarly, we have that $f(t_h)<-1$. Therefore, there exist $r_h\in(s_h,t_h)$, $r_h\to+\infty$, such that $f(r_h)=0$. However, arguing as above, we have that up to a not relabelled subsequence, $v(r_h)$ converges in $L^p(A)$ either to $w$ or to $-w$, that is $f(r_h)$ converges either to $-2$ or to $2$, which is impossible. This contradiction proves \eqref{3}.
  \end{proof}

\begin{Rem}\label{rem:repar}
{\rm The fact that the limit for the path of Lemma \ref{lem:path} exists as $t\to \infty$ allows us to 
reparametrize the path to finite time, preserving continuity. Hence, for 
our purposes we will assume that $v \in C\big([0,1],\M_p(A)\big)$ with 
$ v(1) = u_1$ or $-u_1$. }
\end{Rem}

\begin{proof}[of Theorem \ref{prop:mount pass}]
Recalling that $\Gamma(u_1,-u_1)$ is the set of all continuous 
paths with values in $\mathcal M_p(A)$ joining $u_1$ to $-u_1$, let us define 
$$ \lambda := \inf_{\gamma\in\Gamma(u_1,-u_1)}
  \bigg[\,\max_{w\in\gamma([0,1])} \int_A |\gr w|^p \dx\,\bigg]. $$
Clearly, $\lambda_1(A)\leq\lambda$. Observe that to prove the result it is enough to show that 
 there exists an admissible curve $\gamma \in \Gamma(u_1,-u_1)$ such that 
 \begin{equation}\label{eq:semcont}
\max_{\,t\in[0,1]} \int_A |\gr \gamma(t)|^p \dx =\lambda_2(A).
\end{equation}
Indeed if $\lambda_1(A)$ is not simple from the previous equality we trivially have $\lambda=\lambda_2(A)=\lambda_1(A)$. On the other hand, if $\lambda_1(A)$ is simple, then by Theorem~\ref{thm:Mpass} and Lemma~\ref{utile} we have that $\lambda$ is an eigenvalue; since by Definition~\ref{def:sec eigen} there is no other eigenvalue between $\lambda_1(A)$ and 
$\lambda_2(A)$, from \eqref{eq:semcont} we get $\lambda=\lambda_2(A)$, thus  concluding the proof of the theorem.

$ \textbf{Case 1 :}\ \lambda_1(A)$ is simple. In this case, thanks to Lemma~\ref{lem:first eigen} we may assume with no loss of generality that $u_1\geq0$. Then we set $U:=\{u_1>0\}$ and recall that
by Corollary~\ref{cor:first eigen support}, $U$ is a $p$-quasi connected component of $A$. Denote by $u_2$ a normalized second eigenfunction. 

Assume first that $u_2$ changes sign in $U$, hence $u_2^+$ cannot be an eigenfunction, otherwise by the minimum principle either $u_2^+>0$ or $u_2^+=0$ q.e. in $U$. Similarly $u_2^-$ is not an eigenfunction. In this case the  goal is to construct a continuous curve on $\M_p(A)$ from $u_1$ to $-u_1$, 
 such that the energy $E$ reaches the maximum value
 $\lambda_2(A)$ at a point and and stays below this value elsewhere.
Wee denote, for $t\in[0,1]$, 
 \begin{equation}\label{eq:curves}
  w(t):=\frac{(1-t)u_2^+}{\|u_2^+\|_{L^p(A)}}-\frac{(1-(1-t)^p)^{1/p}u_2^-}{\|u_2^-\|_{L^p(A)}}.
 \end{equation}
Note that $w$ is a curve with values in $\mathcal M_p(A)$ connecting
$ u_2^+/\|u_2^+ \|_{L^p(A)}$ to $-u_2^-/\|u_2^-\|_{L^p(A)}$  and that $E(w(t))=\lambda_2(A)$ for all $t$. Since $u^+_2$ is not an eigenfunction, using Lemma \ref{lem:path} 
and Remark \ref{rem:repar}, we may construct two curves $v_i \in C\big([0,1],\M_p(A)\big)$, $i=1,2$, with the property that $E(v_i(t))\leq\lambda_2(A)$ for all $t\in[0,1]$ and such that $v_1$ connects $u^+_2/\|u^+_2\|_{L^p(A)}$ to $u_1$ and $v_2$ connects  $-u^-_2/\|u^-_2\|_{L^p(A)}$ to $-u_1$. Then, denoting by $\inv{v_1}$ the path $v_1$ covered in the opposite direction, we set 
$$\gamma:=\,\inv{v_1}*\,w\,*v_2,$$
where $*$ denotes the concatenation of two curves.
By this construction, it is evident that $\gamma$ is an admissible curve satisfying \eqref{eq:semcont}. 

Assume now that $u_2$ does not change sign in $U$, say $u_2 \geq0$ in $U$.  By the minimum principle, either $u_2=0$ or $u_2>0$ q.e. in $U$. In the latter case from Lemma \ref{lem:first eigen}, we have that 
$\lambda_2(A) = \lambda_1(U)=\lambda_1(A)$ which is impossible since $\lambda_1(A)$ is simple. Hence, $u_2=0$ q.e. in $U$. Following
Brasco-Franzina \cite{Bra-Fra--hks}, we define a curve $\gamma\in \Gamma(u_1,-u_1)$ by setting
\begin{equation}\label{cbf}
 \gamma(t):= \frac{\cos(\pi t)u_1\,+\,t(1-t)u_2}{\big(|\cos(\pi t)|^p\, +\, t^p(1-t)^p\big)^{1/p}} 
 \end{equation}
for all $t\in[0,1]$. As before, $\gamma$ is an admissible curve satisfying \eqref{eq:semcont}. 

$ \textbf{Case 2 :}$ $\lambda_1(A)$ is not simple. Assume first that $u_1$ is supported in a $p$-quasi connected component $U$ of $A$. Then   there exists another first eigenfunction $v$ which is different from both $u_1$ and $-u_1$. If $v$  is supported in $U$, then by Proposition~\ref{extAH} $v$ must coincide in $U$ either with $u_1$ or $-u_1$. Therefore there exists a $p$-quasi connected component $U'$ of $A$, with $\pcap(U\cap U')=0$ where $v$ is not identically zero. Denote by $u_2$ the restriction of $v$ to $U'$, normalized so to have $L^p$ norm equal to 1. Using Lemma~\ref{lem:quasi res} we have that $u_2$ is still a first eigenfunction. Moreover  \eqref{cbf} provides again a curve satisfying \eqref{eq:semcont}.

Finally if there exist two or more connected components where $u_1$ is not identically zero, let us denote by $U$ one of these connected components and let us set
$$
\gamma(t):= \frac{\cos(\pi t)u_1\chi_U\,+\,a(t)u_1\chi_{A\setminus U}}{\big(|\cos(\pi t)|^p\|u_1\|_{L^p(U)}^p\, +\,|a(t)|^p\|u_1\|_{L^p(A\setminus U)}^p\big)^{1/p}}, 
$$
where $a:[0,1]\to[-1,1]$ is a strictly decreasing smooth function such that $a(0)=1$, $a(1/2)>0$, $a(1)=-1$. Then it is easily checked that   $\gamma$ is again an admissible curve satisfying \eqref{eq:semcont}. 
\end{proof}
We conclude this section with the following simple consequence of Theorem~\ref{prop:mount pass}.
\begin{Cor}\label{fili}
Let $A\subset B$ be two $p$-quasi open sets of finite measure. Then $\lambda_i(B)\leq\lambda_i(A)$ for $i=1,2$.
\end{Cor}
\begin{proof}
The inequality $\lambda_1(B)\leq\lambda_1(A)$ is an immediate consequence of \eqref{fili1}.

To show that $\lambda_2(B)\leq\lambda_2(A)$, let us denote by $u_{1,A}$, $u_{1,B}$ two nonnegative normalized first eigenfunctions of $A$, and $B$ respectively.  Setting, for $t\in[0,1]$
$$
v_1(t)=\big(tu_{1,A}^p+(1-t)u_{1,B}^p\big)^{1/p},\qquad v_2=-v_1(t),
$$
we have, see \cite[Lemma 2.1]{Bra-Fra--hidconv},
$$
\int_\Omega|\nabla v_i(t)|^p\,dx\leq t\int_\Omega|\nabla u_{1,A}(t)|^p\,dx+(1-t)\int_\Omega|\nabla u_{1,B}(t)|^p\,dx\leq\lambda_1(A).
$$
On the other hand, thanks to Theorem~\ref{prop:mount pass} there exists a map $\gamma\in C([0,1],\M_p(A))$ such that $\gamma(0)=u_{1,A}$, $\gamma(1)=-u_{1,A}$ and  \eqref{eq:semcont} holds. Therefore, setting $w=v_1*\gamma*\inv{v_2}$, we have $w\in C([0,1],\M_p(B))$ and by construction
$$
\max_{\,t\in[0,1]} \int_B |\gr w(t)|^p \dx =\lambda_2(A).
$$
From this equality and Theorem~\ref{prop:mount pass} applied to $B$ we then get $\lambda_2(B)\leq\lambda_2(A)$. 
\end{proof}

\section{$\gamma_p$-lower semicontinuity of eigenvalues}\label{sec:gamma cont}

In this section we fix a bounded open set $\Om\subset \R^n$. Henceforth, given a  $p$-quasi open set $A \in \A_p(\Om)$  and a function $u\in W^{1,\,p}_0(A)$ we shall still denote by $u$  its zero extension in $\Omega\setminus A$, which is a function in $W^{1,\,p}_0(\Omega)$.

\subsection{$\gamma_p$-convergence and properties}
We now introduce the $\gamma_p$-convergence of $p$-quasi open sets. 
Differently from the case  $p=2$ considered in \cite{Bu-DM}, 
in the following definition we require the {\it weak} convergence in $W^{1,\,p}$ of the resolvents and not the strong one. Indeed, in view of the nonlinearity of the $p$-Laplacian, 
requiring the strong convergence of the resolvent operators would 
end up in a too strong topology in $\A_p(\Om)$ with very few compact  sets. 
Instead, the definition below provides plenty of compact 
families in $\A_p(\Om)$. However, the drawback is that now the proof of the
lower semicontinuity of the eigenvalues requires a more delicate argument. 
\begin{Def}\label{def:gamma}
{\rm
Let  $A_m, A$ be $p$-quasi open sets 
in $\A_p(\Om)$ for every $m\in\N$. We say that 
the sequence 
$A_m$ {\it $\gamma_p$-converges} to $A$ as $m\to \infty$ and 
we write $A_m \gto A$, if $ \Rsp{A_m}(f) \wto \Rsp{A}(f) $ weakly in $ W^{1,\,p}_0(\Om)$ 
for every $f\in W^{-1,\,p^\prime}\!(\Om)$,  
where the operators $\Rsp{A_m}$ are defined as in  Definition~\ref{def:res}.
}
\end{Def}
The above definition of $\gamma_p$-convergence of $p$-quasi 
open sets is strongly related to a convergence in the space 
$\mathcal M_0^p(\Omega)$ of Borel measures with values in $[0,\infty]$ 
vanishing on sets of zero $p$-capacity introduced by Dal Maso-Murat in \cite{DM-Mu97}. They say that a sequence $\mu_m\in\mathcal M_0^p(\Omega)$ {\it $\gamma$-converges} to a measure $\mu\in\mathcal M_0^p(\Omega)$ if for any $f\in W^{-1,\,p^\prime}\!(\Om)$ the solutions $u_m\in W^{1,\,p}_0(\Omega)$ of the  equations
$$
\int_\Omega|\gr u_m|^{p-2}\gr u_m\cdot\gr\varphi\,dx+ \int_\Omega|u_m|^{p-2}
u_m\varphi\ d\mu_m=\langle f,\varphi\rangle \qquad\text{for all $\varphi\in W^{1,\,p}_0(\Omega)$}
$$
converge weakly in $W^{1,\,p}(\Omega)$ to the solution of the corresponding equation with $\mu_m$ replaced by $\mu$.
It is evident that the Definition~\ref{def:gamma} is equivalent to the $\gamma$-convergence of the measures $\infty_{ A_m}$ to $\infty_{A}$ in the sense of Dal Maso-Murat \cite{DM-Mu97}, where by $\infty_{ A}$ we denote the measure in $\mathcal M_0^p(\Omega)$ defined by
\[
\infty_{A}(B):= \begin{cases}
0 \ \  &\text{if}\ \pcap(B\cap A) = 0,  \\
+\infty \ \ &\text{if}\ \pcap(B\setminus A) > 0
\end{cases}
\]
for all Borel sets $B\subset\Omega$. 
With this observation in mind, the next theorem follows immediately from a general 
result of Dal Maso and Murat, see \cite[Th.~6.3]{DM-Mu97}.
\begin{Thm}\label{DM6.3}
$A_m \gto A$ in $\A_p(\Om)$ if and only if  $\Rsp{A_m}(1) \wto \Rsp{A}(1) $ weakly in $ W^{1,\,p}_0(\Om)$.
\end{Thm}
The following theorem is also contained in the above mentioned paper, see 
 \cite[Th.~6.8]{DM-Mu97}. 
\begin{Thm}\label{DM6.8}
Let $A_m,A\in \A_p(\Om)$ be such that  $A_m \gto A$. Then for every $f\in W^{-1,\,p^\prime}\!(\Om)$ we have that $ \Rsp{A_m}(f) \to \Rsp{A}(f) $ strongly in $ W^{1,r}_0(\Om)$ for all $1\leq r<p$.
\end{Thm}
Now we show that, if the underlying quasi open sets $\gamma_p$-converge, then 
the limit of the sequence of eigenvalues is still an eigenvalue and the corresponding eigenfunctions converge strongly in $W^{1,r}(\Omega)$ for all $1\leq r<p$.  
\begin{Prop}\label{prop:gamma stability}
Let $A_m\in\A_p(\Om)$ be a sequence of $p$-quasi open sets and let $\lambda_m $ be, for every $m \in \N$,  an eigenvalue of the $p$-Laplacian in $A_m$ with a normalized 
eigenfunction $u_m\in W^{1,\,p}_0(A_m)$. 
If there exist $A\in\A_p(\Om)$ and $\lambda\in\R$ such that $A_m \gto A$ 
and $\lambda_m \to \lambda$ as $m \to \infty$, then $\lambda$ is an eigenvalue 
of the $p$-Laplacian in $A$ (hence $\lambda>0$) and the eigenfunctions $u_m$ converge in $W^{1,r}(\Omega)$, up to a subsequence, to an eigenfunction $u$ of $\lambda$, for all $1\leq r<p$. 
\end{Prop}
\begin{proof}
Since $u_m$ is a normalized eigenfunction of  $\lambda_m$ and $\lambda_m\to \lambda$, 
the sequence  $ \{u_m\}$ is  bounded 
in $ W^{1,\,p}_0(\Om)$. Therefore there exists a function 
$u\in W^{1,\,p}_0(\Om)$ such that, up to a not relabelled subsequence,  $ u_m \wto u$ in $W^{1,\,p}(\Omega)$, 
$ u_m \to u$ in $L^p(\Om)$ and a.e. in $\Omega$. Hence $\|u\|_{L^p(\Om)}=1$. 
Let  us now set
 $$ v_m := \Rsp{A_m}(\lambda |u|^{p-2}u).$$
We claim that
\begin{equation}\label{gamsta1}
\lim_{m\to \infty}\|u_m-v_m\|_{W^{1,\,p}_0(\Om)}=0. 
\end{equation}
Since $A_m \gto A$, by Theorem~\ref{DM6.8} we have that
$ v_m \to \Rsp{A}(\lambda |u|^{p-2}u) $ in $W^{1,r}(\Omega)$ for all $1\leq r<p$. Then  the claim \eqref{gamsta1} yields that
$ u_m \to \Rsp{A}(\lambda |u|^{p-2}u)$ in $W^{1,r}(\Omega)$ for all $1\leq r<p$. But since   $u_m\wto u$ weakly in $W^{1,\,p}(\Omega)$, 
we conclude that $ u = \Rsp{A}(\lambda |u|^{p-2}u)$. Thus $u\in W^{1,\,p}_0(A)$, $\lambda$ is an eigenvalue of $A$ and $u$ is a corresponding eigenfunction. This concludes the proof, provided we show that the claim \eqref{gamsta1} holds.

To this end, note that $u_m$ and $v_m$ satisfy the 
following equations in $A_m$. 
\begin{align*}
 -\dv(|\gr u_m|^{p-2}\gr u_m)\ &=\ \lambda_m |u_m|^{p-2}u_m; \\ 
 -\dv(|\gr v_m|^{p-2}\gr v_m)\ &=\ \lambda\, |u|^{p-2}u.
\end{align*}
Testing both equations by the  function $ u_m-v_m$ and subtracting the resulting equalities, we obtain 
\begin{align*}
&\int_{A_m}\Big( |\gr u_m|^{p-2}\gr u_m - |\gr v_m|^{p-2}\gr v_m\Big)\cdot\big( \gr u_m-\gr v_m\big)\dx\\
 &\qquad\qquad\qquad =\ \int_{A_m}\Big [\lambda_m|u_m|^{p-2} u_m - \lambda| u|^{p-2}u\Big](  u_m- v_m)\,dx.
\end{align*}
By the a.e. convergence of $u_m$ to $u$, using a well known variant of the Lebesgue dominated convergence theorem, we get that  the sequence $\lambda_m|u_m|^{p-2}u_m$ converges in $L^{p^\prime}\!(\Omega)$ to $\lambda|u|^{p-2}u$. Since the sequence $u_m-v_m$ is bounded in $L^p(\Omega)$, we get that the right hand side of the above equality converges to zero. Then, from 
Lemma \ref{lem:ineq} we get immediately that  
$$ \lim_{m\to\infty}\int_\Om |\gr u_m-\gr v_m|^p\dx=0. $$
This proves the claim \eqref{gamsta1}, thereby completing the proof of the lemma.
\end{proof}

\subsection{$\gamma_p$-lower semicontinuity of eigenvalues}\noindent
Now we investigate the behavior of the $p$-Laplacian eigenvalues with respect to the $\gamma_p$-convergence of quasi open sets. The case of  
the first eigenvalue is easy to deal with.
\begin{Cor}[(Lower semicontinuity of $\lambda_1$)]\label{prop:gam cont}
Let $A_m,A\in \mathcal A_p(\Omega)$ be such that $A_m\gto A $.  Then 
\begin{equation}\label{prop:gam cont1}
\lambda_1(A)\,\leq\, \liminf_{m\to\infty}\lambda_1(A_m). 
\end{equation}
\end{Cor}
\begin{proof}
Without loss of generality we may assume that the above $\liminf$  is indeed a limit, say $\lambda$, and that $\lambda$ is finite.
From Proposition \ref{prop:gamma stability}, we know that $\lambda$ is an eigenvalue of $A$. As $\lambda_1(A)$ is the 
minimum of all eigenvalues, $\lambda_1(A)\leq\lambda$ and \eqref{prop:gam cont1} follows.
\end{proof}
The proof of lower semicontinuity for the second eigenvalue, is more involved. To this end we   have to use both the result stated in  Theorem~\ref{prop:mount pass} and a construction  based  on Lemma~\ref{lem:path}. 

\begin{Prop}[(Lower semicontinuity of $\lambda_2$)]\label{prop:lsc}
Let $A_m,A\in \mathcal A_p(\Omega)$ be such that $A_m\gto A $. Then 
\begin{equation}\label{prop:lsc1}
 \lambda_2(A) \,\leq\, \liminf_{m\to \infty}\lambda_2(A_m). 
 \end{equation}
\end{Prop}

\begin{proof}
If $\lambda_1(A)$ is not simple then \eqref{prop:lsc1} follows at once from Definition \ref{def:sec eigen} and 
Corollary \ref{prop:gam cont}.
Hence in the rest of the proof we assume that $\lambda_1(A)$ is simple. 

Let  $u_1$be  the first nonnegative normalized eigenfunction of $\lambda_1(A)$. Without loss of generality we may assume that the $\liminf$ in  \eqref{prop:lsc1}  is a limit and that it is finite. Then
\begin{equation}\label{prop:lsc2}
\lambda_2(A_m) \to \lambda\quad\text{as}\ m\to\infty 
 \end{equation}
and by Proposition~\ref{prop:gamma stability}, $\lambda$ is an eigenvalue.
For every $m\in\N$, let $u_{1,m}$ be a normalized nonnegative eigenfunction of $\lambda_1(A_m)$ supported in a $p$-quasi connected component $U_m$ of $A_m$. Note that such eigenfunction always exists thanks to Lemma~\ref{lem:quasi res}. Moreover, by Proposition~\ref{prop:gamma stability}, we may assume that the sequence $u_{1,m}$ converges weakly in $W^{1,\,p}(\Omega)$ to an eigenfunction $w$. 

For every $m$ let us denote by $u_{2,m}$ a normalized eigenfunction for $\lambda_{2,m}(A_m)$. We now construct a suitable sequence of curves $\gamma_m\in W^{1,\,p}\big([0,1], \M_p(A_m)\big)$ with endpoints $\pm u_{1,m}$, by considering the following cases.

$ \textbf{Case 1 :}$  $u_{2,m}$ changes sign in $U_m$.
\\ 
 We first construct a continuous curve $w_m$ on $\M_p$ from $u_{2,m}^+/\|u_{2,m}^+ \|_{L^p(\Om)}$ to $-u_{2,m}^-/\|u_{2,m}^-\|_{L^p(\Om)}$, in a  different way from what we did in \eqref{eq:curves}.  We set 
 $$
 w_{m}(t):=\frac{\frac{(1-t)u^+_{2,m}}{\|u^+_{2,m}\|_{L^p(\Om)}}-\frac{tu^-_{2,m}}{\|u^-_{2,m}\|_{L^p(\Om)}}}{((1-t)^p+t^p)^{1/p}}.
 $$
Note that  the above construction yields that  there exists a constant $C$ independent of $m$ such that for all $m\in\N$
\begin{equation}\label{prop:lsc3}
\int_0^1\|w_{m}'(t)\|_{L^p(\Omega)}^p\,dt\,\leq  C 
\end{equation}
and furthermore that $E(w_{m}(t))=\lambda_2(A_m)$ for all $t\in[0,1]$. Now, since $u^+_{2,m}$ and $u^-_{2,m}$ are not eigenfunctions, using Lemma \ref{lem:path} 
and Remark \ref{rem:repar}, we can find two curves $v_{i,m}\in W^{1,\,p}\big([0,1], \M_p(A_m)\big)$ with $i=1,2$, which have the property that $E(v_{i,m}(t))\leq\lambda_2(A_m)$ for all $t\in[0,1]$ and such that $v_{1,m}$ connects $u^+_{2,m}/\|u^+_{2,m}\|_{L^p(A_m)}$ to $u_{1,m}$ and $v_{2,m}$ connects  $-u^-_{2,m}/\|u^-_{2,m}\|_{L^p(A_m)}$ to $-u_{1,m}$. Then, denoting by $\inv{v_{1,m}}$ the path $v_{1,m}$ covered in the opposite direction, we set 
$$
\gamma_m:=\,\inv{v_{1,m}}*\,w_m\,*v_{2,m}.
$$
Note that from \eqref{prop:lsc3} and the second inequality in \eqref{1}, we have that
$$
\int_0^1\|\gamma_{m}'(t)\|_{L^p(\Omega)}^p\,dt\,\leq C.
$$

$ \textbf{Case 2 :}$ $u_{2,m}$ has  constant sign in $U_m$.

\noindent
In this case, arguing as in the proof of Theorem~\ref{prop:mount pass}, we may always find another second eigenfunction, still denoted by $u_{2,m}$, whose support is disjoint from $U_m$ up to a set of zero $p$-capacity.
Thus 
we define the curves $\gamma_m\in \Gamma(u_{1,m},-u_{1,m})$ by setting
$$
 \gamma_m(t)=\frac{\cos(\pi t)u_{1,m}\,+\,t(1-t)u_{2,m}}{\big(|\cos(\pi t)|^p\, +\, t^p(1-t)^p\big)^{1/p}} \qquad t\in[0,1].
$$
Then it is easily checked that $E(\gamma_m(t))\leq\lambda_2(A_m)$ for all $m$ and $t$ and that also in this case there exists a constant $C$ such that for all $m$
$$
\int_0^1\|\gamma_{m}'(t)\|_{L^p(\Omega)}^p\,dt\,\leq C.
$$

Combining the two cases, we  conclude that there exists a sequence  $\gamma_m\in W^{1,\,p}\big([0,1], \M_p(A_m)\big)$ of curves with endpoints 
 $\pm u_{1,m}$, such that for all $m\in\N$, we have 
$$
\int_0^1\|\gamma_m'(t)\|_{L^p(\Omega)}^p\,dt\,\leq C\quad\text{and}\quad E(\gamma_m(t))\leq\lambda_2(A_m),\quad\text{ for all $t>0$}.
$$
Now we prove that $\lambda\geq\lambda_2(A)$, where $\lambda$ is the limit in \eqref{prop:lsc2}.

To this end we argue by contradiction, assuming that $\lambda<\lambda_2(A)$. Then we have that $\lambda=\lambda_1(A)$ and that $u_{1,m}$ converges to $u_1$.
Therefore, using Arzel\`a-Ascoli theorem as in Step 1 of the proof of Theorem~\ref{prop:mount pass}, we conclude that there exists  $\gamma \in W^{1,\,p}\big([0,1], L^p(\Omega)\big)$ such that,  up to a  subsequence,
$\gamma_m(t) \to \gamma(t) $ in $ L^p(\Om)$ and weakly in $W^{1,p}_0(\Omega)$ for all $t\in[0,1]$.

Note also that $\gamma(t)\in W^{1,\,p}_0(A)$ for all $t\in[0,1]$, hence $\gamma \in W^{1,\,p}\big([0,1], \M_p(A)\big)$. Indeed,  the functions $w_{A_m}= \Rsp{A_m}(1)\wto\Rsp{A}(1)$ weakly in $W^{1,p}_0(\Omega)$. Thus, Lemma~\ref{support} below yields that $\gamma(t)=0$ q.e. in $\Omega\setminus A$. 

 Since the endpoints of $\gamma$ are $\pm u_1$, from Theorem~\ref{prop:mount pass} we conclude that 
\begin{align*}
\lambda_2(A) \leq \max_{t\in[0,1]}\int_A |\gr\gamma(t)|^p\dx
\,\leq\, \liminf_{m\to \infty}\bigg[\max_{t\in[0,1]}\int_A |\gr\gamma_m(t)|^p\dx\bigg]
\leq\lim_{m\to \infty}\lambda_2(A_m)=\lambda_1(A), 
\end{align*}
which is impossible since $\lambda_1(A)$ is simple. This contradiction concludes the proof.
\end{proof}

\section{A shape optimization problem}\label{sec:existence}
In this section we prove the following theorem, which is the $p$-Laplacian 
counterpart of the existence theorem of Buttazzo-Dal Maso \cite{Bu-DM}. 
With this theorem in hand, 
Theorem~\ref{thm:mainthm2} follows at once, thanks to Corollaries~\ref{fili} and \ref{prop:gam cont} and to Proposition~\ref{prop:lsc}.

\begin{Thm}\label{thm:existence}
Let  $\Om \subset \R^n$ be a bounded open set and 
 $F : \A_p(\Om) \to \R$ be a decreasing function, lower semicontinuous with respect to 
$\gamma_p$-convergence. Then the minimization problem 
\begin{equation}\label{eq:min prob}
\min\set{F(A)}{A\in \A_p(\Om),\ |A|=c}, 
\end{equation}
where $\ 0< c \leq |\Om| $, always has a solution.
\end{Thm}

For every $A\in \A_p(\Omega)$, we set $w_A := \Rsp{A}(1)$.  We claim that   $w_A$ is a subsolution of the equation $-\lap_pu= 1$
 in $\Om$. This is the content of  the following 
lemma. The proof is similar to the one given in  \cite[p. 190]{Bu-DM}, for the case $p=2$. For the reader's convenience, we provide the proof in the Appendix.
\begin{Lem}[(Comparison principle)]\label{lem:p subtorsion} Let $\Om\subset\R^n$ be a bounded open set and  $A\in \A_p(\Om)$. Then $ w_A = \Rsp{A}(1)$ is a   subsolution 
 of the equation $-\lap_pu= 1$
 in $\Om$, i.e.
 \begin{equation}\label{eq:p subtorsion}
\int_\Om|\gr w_A|^{p-2}\gr w_A\cdot\gr \varphi\,dx\leq \int_\Om\varphi\,dx 
\end{equation}
for all nonnegative functions $\varphi\in W^{1,\,p}_0(\Om)$. Moreover, if 
$ w \in W^{1,\,p}_0(\Om)$ is another  subsolution satisfying \eqref{eq:p subtorsion} and  such  that    
$ w \leq 0$ q.e. on $\ \Om \mns A$, then $ w_A \geq w$ q.e. in $\ \Om$.
\end{Lem}

\subsection{The main construction}
 Following \cite{Bu-DM}, we now fix a closed convex subset $K \subset W^{1,\,p}_0(\Om) $ defined by imposing an obstacle condition. Precisely, we set
 $$
 K:  = \set{w\in W^{1,\,p}_0(\Om)}{w\geq 0,\ -\lap_p w-1\le 0}. 
 $$
From  Lemma~\ref{lem:p subtorsion} we have that $w_A\in K$ for every $A\in \A_p(\Om)$. Moreover, if $w\in K$, multiplying
the inequality $-\lap_p w\le 1 $ by $w$, we get
\begin{equation}\label{equ 1}
\int_{\Omega} |\gr w|^p\,dx\leq  \int_{\Omega} w \, dx.
\end{equation}
Thus, by the Poincar\'e inequality we conclude that $K$ is bounded in $W^{1,\,p}_0(\Om)$ and compact in $L^p(\Omega)$. At this point, still following
 \cite{Bu-DM}, we define an auxiliary functional $G:K \to \R$ and 
reduce  the proof of the existence of a minimizer of the problem \eqref{eq:min prob} to  showing  the existence of  a minimizer of the problem
$ \min\set{G(w)}{w\in K,\, |\{w>0\}|\le c }$. Before defining  $G$, we list the properties that we require from this functional. \\

\begin{enumerate}[(i)]
\item $G$ is decreasing, i.e. for every $u,\,v\in K$ with $u\le v$ q.e. in $\Omega$,  then $G(u)\ge G(v)$; \label{i}
\item $G$ is lower semicontinuous on $K$ with respect to the strong topology of $L^p(\Omega)$; \label{ii}
\item $G(w_A)=F(A)$ for every $A\in \A_p(\Omega)$, where $w_A = \Rsp{A}(1)$. \label{iii}\\
\end{enumerate}
\begin{Def}\label{def:G}
{\rm
For every $w\in K$ we set $J(w) = \inf\set{F(A)}{A\in \A_p(\Omega),\, w_A\le w}$ and define $G$ as the $L^p(\Omega)$-lower semicontinuous envelope of $J$. In other words
$G:K \to \R$ is defined by setting 
$$
G(w) =\inf\set{\liminf_{h\to \infty}J(w_h)}{w_h \in K,\ w_h \to w\ \text{in}\ L^p(\Om)}.
$$
}
\end{Def}
We now show that the functional $G$  satisfies  properties 
\ref{i}, \ref{ii} and \ref{iii}. The verification of the first two is relatively easy, as shown in 
the next lemma.
\begin{Lem}\label{i and ii}
The functional $G$  satisfies \ref{i} and \ref{ii}. 
\end{Lem}

\begin{proof}
 Property \ref{ii} is an immediate consequence of the definition of $G$. 
 
 In order to show \ref{i}, first note that $J$ is decreasing. Fix $u,\,v\in K$ with $u\le v$ q.e. in $\Omega$. Then by the definition of $G(u)$, there exists a sequence of $\{u_h\}\in K$ such that $u_h$ converge strongly in $L^p(\Omega)$ and pointwise a.e. to $u$ and 
$$G(u)=\lim_{h\to \infty}J(u_h).$$
Set $v_h=\max\{v,\,u_h\}$. Then $v_h\in K$ since the maximum of two  subsolutions of the equation $-\lap_pw=1$ is still a  subsolution. Moreover $v_h\to v$ in  in $L^p(\Omega)$ by the dominate convergence theorem and the assumption that $u\le v$ q.e. in $\Omega$. Hence
$$G(v)\le \liminf_{h\to \infty}J(v_h)\le \lim_{h\to \infty}J(u_h)=G(u),$$
which proves \ref{i}. 
\end{proof}

Property \ref{iii} will  follow from Lemma \ref{lower semicontinuity} below. First, 
we recall the definition of   $\Gamma$-convergence and prove the auxiliary Lemma~\ref{support}. This lemma is the $p$-Laplacian counterpart of Lemma 3.2 in \cite{Bu-DM}
and the main result of this section. 
\begin{Def}[($\Gamma$-convergence)]\label{def:Gamma func}
{\rm 
 Let $\Phi_h, \Phi :L^p(\Om) \to \R\cup\{+\infty\}$. We say that the functionals $\Phi_h$ {\it $\Gamma$-converge  in $L^p(\Omega)$ to $\Phi$} if 
 the following two conditions are satisfied:
 \begin{enumerate}
 \item for every $u_h\in L^p(\Om)$ such that $u_h\to u\in L^p(\Om)$, then 
$$\Phi(u)\le \liminf_{h\to \infty} \Phi_h(u_h);$$
  \item for every $u\in L^p(\Om)$ there exists a sequence $u_h\to u$ such that
$$\Phi(u)\ge \limsup_{h\to \infty} \Phi_h(u_h).$$
 \end{enumerate}
This convergence shall be denoted by 
 $\Phi_h\Gto\Phi$. 
 }
\end{Def}
With this definition in hand we are  ready to prove the key lemma of this section.
\begin{Lem}\label{support}
Let $\{A_h\}$ be a sequence in $\A_p(\Omega)$ such that the functions $w_{A_h}$ 
converge weakly to $w$ in $W^{1,\,p}_{0}(\Omega)$ 
and let $u_h$ be a sequence in $W^{1,\,p}_{0}(\Omega)$ such that $u_h=0$ q.e. 
in $\Omega\setminus A_h$. If $u_h \wto u$ in  $W^{1,\,p}_{0}(\Omega)$, then $u=0$ 
q.e. in $\{w=0\}$. 
\end{Lem}

\begin{proof}
Following \cite{Bu-DM}, we define the functionals
$\Phi_h : W^{1,\,p}(\Om) \to \R\cup\{+\infty\}$ by setting 
$$
\Phi_h(v)=\begin{cases}
\displaystyle\frac1p\int_{A_h} |\gr v|^p\, dx \ \  &\text{for}\ \ v\in W^{1,\,p}_{0}(A_h), \\
+\infty \ \ & \text{otherwise}.\ \ 
\end{cases}
$$
By a general compactness result, 
see \cite[Th. 4.18 and Prop. 4.11]{DMM1981}, 
there exists a functional $\Phi :W^{1,p}(\Omega) \to \mathbb R\cup\{+\infty\}$ such that, up to a not relabelled subsequence, $\Phi_h \Gto \Phi$.  Let 
 $\D(\Phi_h)$ and $\D(\Phi)$ be the effective domains of $\Phi_h$ and $\Phi$, 
 respectively, see \eqref{effe}. 
 
Observe that if  $u_h \in W^{1,\,p}_{0}(A_h)$ and $u_h \wto u$ 
weakly in $W^{1,\,p}_{0}(\Omega)$  then by condition (i) in  
Definition \ref{def:Gamma func}, we  get that
$$\Phi(u)\leq \liminf_{h\to \infty} \Phi_h(u_h)
=\liminf_{h\to \infty}\frac1p\int_{\Omega} |\gr u_h|^p\,dx< \infty,$$
hence $u\in \D(\Phi)$. Conversely, for
every $v \in \D(\Phi)$, from Definition~\ref{def:Gamma func} it follows that 
there exists a sequence $v_h\in W^{1,\,p}(\Omega)$ converging strongly in $L^p(\Omega)$ to $v$ 
such that $ \Phi_h(v_h) \to \Phi(v) $. Thus the sequence $v_h$ actually converges weakly to $v$ in $W^{1,\,p}(\Omega)$.

Note that $\Phi$ is convex, actually strictly convex on $\D(\Phi)$. The latter property is proved by observing that if 
$\Phi((z+z')/2)=1/2(\Phi(z)+\Phi(z'))<\infty$, then denoting by $z_h$ and $z_h'$ two sequences, converging in $L^p(\Omega)$ to $z$ and $z'$ respectively, such that $\Phi(z)=\displaystyle\lim_{h\to\infty}\Phi(z_h)$ and $\Phi(z')=\displaystyle\lim_{h\to\infty}\Phi(z_h')$, then we have also that $\Phi((z+z')/2)=
\displaystyle\lim_{h\to\infty}\Phi_h((z_h+z_h')/2)$. Therefore we have in particular that
$$
\lim_{h\to\infty}\bigg[\frac{1}{2}\int_\Omega(|\nabla z_h|^p+|\nabla z_h'|^p)\,dx-\int_\Omega\Big|\frac{\nabla z_h+\nabla z'_h}{2}\Big|^p\,dx\bigg]=0,
$$
from which we easily conclude, using the Clarkson's inequality, that $\nabla z_h-\nabla z_h'\to0$  in $L^p(\Omega)$ and  thus $z=z'$.

In view of the above remarks, the proof of the lemma will be achieved if we prove the following
 \noindent
\\
$\textbf{Claim :}$ For every $v\in \D(\Phi)$, one has $v=0$ q.e. on $\{w=0\}$. \\

To prove the claim, it is enough to assume that $v\in Int(\D(\Phi))$, since the general case follows by approximation. Since  $v\in Int(\D(\Phi))$, the subdifferential $\partial\Phi(v)$ is not empty, hence we may fix $f \in \partial\Phi(v)$. Denoting by
$\inp{.}{.}$  the duality action between $W^{-1,\,p^\prime}(\Omega)$ and 
$ W^{1,\,p}_{0}(\Omega)$, we have 
$$
v=\text{argmin}\set{\Phi(z)-\inp{f}{z}}{z\in W^{1,\,p}_{0}(\Omega)}.
$$
Let us show  that $v$ is the weak limit in $W^{1,\,p}_0(\Omega)$ of a sequence of minimizers of 
the functional $\Phi_h$, i.e., if 
$$
v_h=\text{argmin}\set{\Phi_h(z)-\inp{f}{z}}{z\in W^{1,\,p}_{0}(\Omega)},
$$
then $v_h \rightharpoonup v$ in $W^{1,\,p}_0(\Omega)$. To this end, observe that up to a subsequence
 $v_h \wto \tilde{v}$ weakly in
 $W^{1,\,p}_0(\Omega)$ and that 
$ \tilde{v}\in \D(\Phi)$. 
On the other hand from (2) of Definition~\ref{def:Gamma func}  there exists $w_h \in W^{1,\,p}_{0}(A_h)$ such that $w_h$ converge in $L^p(\Omega)$ to $v$ and $\Phi_h(w_h)\to\Phi(v)$. Moreover, passing possibly to a not relabelled subsequence, we may assume that $w_h\wto v$ weakly in $W^{1,\,p}_0(\Om)$. Thus, by  lower semicontinuity, using the minimality of $v_h$, we get
\begin{align*}
\Phi(\tilde{v}) - \inp{f}{\tilde{v}} \leq \liminf_{h\to \infty}\Phi_h(v_h) - \inp{f}{v_h}
\leq \lim_{h\to \infty}\Phi_h(w_h) - \inp{f}{w_h}\leq \Phi(v) - \inp{f}{v}.
\end{align*}
Therefore by uniqueness  we have $\tilde{v} = v$  and that, up to another not relabelled subsequence, $v_h\wto v$ in $W^{1,\,p}(\Omega)$. At this point, a standard compactness argument shows the convergence of the whole sequence $v_h$.

We now fix $\eps\in(0,1)$ and a function $f^{\eps}\in L^{\infty}(\Omega)$ such 
that $\|f-f^{\eps}\|_{W^{-1,\,p^\prime}(\Omega)}\le \eps$. Then, for every $h$ we denote by  $v_h^{\eps}$ the function $v_h^{\eps}: = \Rsp{A_h}(f^\eps)$. Testing the equations satisfied by $v_h^{\eps}$ and $v_h$ with the  function $v_h^\eps - v_h$ 
and subtracting the two resulting equalities, we have
\begin{equation}\label{eq:vheps}
  \int_{\Om}\big( |\gr v_h^\eps|^{p-2}
  \gr v_h^\eps - |\gr v_h|^{p-2}\gr v_h\big)\cdot( \gr v_h^\eps-\gr v_h)\dx
   =\langle f^\eps-f\,,\ v_h^\eps-v_h\rangle .
\end{equation}
If $p\geq2$, we recall \eqref{eq:str mon1} which combined with \eqref{eq:vheps},  
followed by 
a standard use of Young's inequality and Poincar\'e inequality, yields 
that for some $c(p)>0$
\begin{equation}\label{eq:vheps1}
  \int_\Om |\gr v_h^\eps -\gr v_h|^p\dx
  \leq c(p)\|f-f^{\eps}\|_{W^{-1,\,p^\prime}(\Omega)}^\frac{p}{p-1}
  \leq c(p)\,\eps^\frac{p}{p-1}.
\end{equation}
If $1<p<2$, we recall \eqref{eq:str mon2}, which combined with \eqref{eq:vheps},  
yields the following.
 \begin{align*}
  \int_\Om |\gr v_h^\eps -\gr v_h|^p\dx 
  \ &\leq \big| \langle f^\eps-f\,,\ v_h^\eps-v_h\rangle \big|^\frac{p}{2}
  \Big(\int_\Om (|\gr v_h^\eps|^p+|\gr v_h|^p)\dx\Big)^{1-\frac{p}{2}}\\
  &=\ \big| \langle f^\eps-f\,,\ v_h^\eps-v_h\rangle \big|^\frac{p}{2}
  \Big(\inp{f^\eps}{v_h^\eps} + \inp{f}{v_h}\Big)^{1-\frac{p}{2}}.
 \end{align*}
Since the sequence $v_h$ is bounded in $W^{1,\,p}(\Om)$, we easily get
\begin{equation*}
 \begin{aligned}
  \int_\Om |\gr v_h^\eps -\gr v_h|^p\dx 
  \ &\leq \|f-f^{\eps}\|_{W^{-1,\,p^\prime}(\Omega)}^\frac{p}{2}
  \|v_h^\eps-v_h\|_{W^{1,\,p}_0(\Omega)}^\frac{p}{2}
  \Big(\inp{f^\eps}{v_h^\eps-v_h} + \inp{f+f_\eps}{v_h}\Big)^{1-\frac{p}{2}}\\
  &\leq C\,\eps^\frac{p}{2}
  \Big(\, \|v_h^\eps-v_h\|_{W^{1,\,p}_0(\Omega)} \,+\, 
  \|v_h^\eps-v_h\|_{W^{1,\,p}_0(\Omega)}^\frac{p}{2}\,\Big),
 \end{aligned}
\end{equation*}
for some constant $C$ depending on $\|f\|_{W^{-1,\,p^\prime}(\Omega)}$ and $p$, but  independent of $h$ and $\eps$. Thus, from Young's inequality and Poincar\'e inequality, we have 
\begin{equation}\label{eq:vheps2}
 \begin{aligned}
  \int_\Om |\gr v_h^\eps -\gr v_h|^p\dx 
  & \leq \frac{1}{2}\, \int_\Om|\gr v_h^\eps-\gr v_h|^p\,dx
  + C\big(\,\eps^\frac{p^2}{2(p-1)}\, +\, \eps^p\,\big)\\
  &\leq \frac12\int_\Om |\gr v_h^\eps -\gr v_h|^p\dx 
  + C\eps^p,
 \end{aligned}
\end{equation}
for $0<\eps<1$. Thus, from \eqref{eq:vheps1} and \eqref{eq:vheps2} we may conclude that there exists a positive constant $C$ depending only on $\Omega, \|f\|_{W^{-1,\,p^\prime}(\Omega)}$ and $p$ but independent of $h$ and $\eps$,  such that 
for every $1<p<\infty$, 
$$\|v_h -v_h^{\eps}\|_{W_0^{1,\,p}(\Omega)}
\leq C\eps^{\frac1p}.
$$
Note that for every $\eps\in(0,1)$ there exists a not relabelled subsequence $v^\eps_h$ converging weakly and a.e. to a function $v^\eps\in W^{1,\,p}_0(\Omega)$. Then, from the previous inequality, we have by  lower semicontinuity 
$$\|v -v ^{\eps}\|_{W_0^{1,\,p}(\Omega)}\leq C\eps^{\frac1p}. $$
Note that by the  comparison principle (Lemma~\ref{lem:p subtorsion}) we have 
$|v_h^{\eps}|\le \|f^\eps\|_{L^\infty(\Om)}^{1/(p-1)}\,w_{A_h}$ q.e.\ in $\Omega$, 
hence $|v ^{\eps}|\le \|f^\eps\|_{L^\infty(\Om)}^{1/(p-1)}\,w$ a.e.\ in $\Omega$. In particular, we have that the precise representative of  $v^\eps$ is $0$ q.e. in the set where the precise representative of $w$ vanishes. Then, the claim follows from the strong convergence in $W^{1,\,p}_0(\Omega)$ of $v^\eps$ to $v$. Hence the proof of the lemma is completed.
\end{proof}

\subsection{Proof of Theorem \ref{thm:existence}}
Now we have all the ingredients to prove the existence theorem. The proof 
follows  the one in \cite{Bu-DM} with   some extra difficulties due to the nonlinearity of the $p$-Laplacian.

\begin{Lem}\label{lower semicontinuity}
Let $w_h$ be a sequence in $K$ converging in $L^p(\Omega)$ to $w_A$ for some $A\in \A_p(\Omega)$.
Then 
$$
F(A)\leq \liminf_{h\to \infty}J(w_h).
$$
\end{Lem}
\begin{proof} Without loss of generality we may assume that the above $\liminf$  is indeed a finite limit.
From Definition~\ref{def:G} we have  there exists $A_h\in \mathscr A_p(\Omega)$ such that $w_{A_h}\le w_h$ and  
$$F(A_h)\leq J(w_h)+ 1/h. $$
Thanks to \eqref{equ 1}, the sequence $\{w_{A_h}\}$ is bounded 
in $W^{1,\,p}_0(\Omega)$, and hence there exists $w\in W^{1,\,p}_0(\Omega)$ such that, up a not relabelled subsequence, 
 $w_{A_h}\wto w$ weakly in $W^{1,\,p}_0(\Omega)$, strongly in $L^p(\Omega)$ and a.e. in $\Omega$.
Therefore since $w_{A_h}\le w_h$ a.e., 
we have that also $w\le w_A$ a.e. in $\Omega$.

Let us fix $\eps>0$ and set $A^{\eps}:=\{w_A>\eps\}.$ Clearly $(w_A-\eps)^+\in W^{1,\,p}_0(A^\eps)$. Passing possibly to another not relabelled subsequence we may assume that the functions $w_{A_h\cup A^\eps}$ converge weakly in $W^{1,\,p}(\Omega)$ and strongly in $L^p(\Omega)$ to a function $w^\eps$. Define now
$v^{\eps}:=1-  \min\{w_A,\,\eps \}/\eps.$
Then  $v^{\eps}\in W^{1,\,p}(\Omega)$ and   $0\le v^{\eps}\le 1$ q.e. in $\Omega$, $v^{\eps}=0$ q.e. in $A^{\eps}$,  $v^{\eps}=1$ q.e. in $\Omega\setminus A$. 
Now set
$$u_h=\min\{v^{\eps},\,w_{A_h\cup A^{\eps}}\}.$$ 
Then  $u_h=0$ q.e. on $\Omega\setminus A_h,$ and, up to another not relabelled subsequence,  $u_h$ converge weakly in $W^{1,\,p}_0(\Om)$ and strongly in $L^p(\Om)$ to $\min\{v^{\eps},\,w^\eps\}$. By Lemma~\ref{support} we conclude that $\min\{v^{\eps},\,w^\eps\}=0$ q.e. in $\{w=0\}$ and in particular that $\min\{v^{\eps},\,w^\eps\}=0$ q.e. in $\Omega\setminus A$. In turn, recalling that $v^\eps=1$ in $\Omega\setminus A$ we have that $w^\eps=0$ q.e. in $\Omega\setminus A$, hence $w^\eps\in W^{1,\,p}_0(A)$.

Now, let us take a sequence $\eps_i$ converging to zero and such that $w^{\eps_i}$ converge strongly in $L^p(\Omega)$ and weakly in $W^{1,\,p}_0(\Om)$ to some function $\widetilde w\in W^{1,\,p}_0(A)$. By a standard diagonal argument we may find a subsequence $w_{A_{h_i}\cup A_{\eps_i}}$  also converging to $\widetilde w$ in $L^p(\Omega)$ and weakly in $W^{1,\,p}_0(\Om)$. By the minimality of $w_{A_{h_i}\cup A^{\eps_i}}$, since $(w_A-\eps_i)^+\in W^{1.p}_0(A_{h_i}\cup A^{\eps_i})$ we have
\begin{equation}\label{nuova1}
\frac1p\int_{\Om}|\gr w_{A_{h_i}\cup A^{\eps_i}}|^p\,dx-\int_{\Om}w_{A_{h_i}\cup A^{\eps_i}}\,dx\leq \frac1p\int_{\Om}|\gr(w_A-\eps_i)^+|^p\,dx-\int_{\Om}(w_A-\eps_i)^+\,dx.
\end{equation}
Passing to the limit, by lower semicontinuity we get that 
\begin{equation}\label{nuova2}
\frac1p\int_{\Om}|\gr\widetilde w|^p\,dx-\int_{\Om}\widetilde w\,dx\leq \frac1p\int_{\Om}|\gr w_A|^p\,dx-\int_{\Om}w_A\,dx.
\end{equation}
Thus, since $\widetilde w\in W^{1,\,p}_0(A)$, by uniqueness we conclude that $\widetilde w=w_A$. Moreover, combining \eqref{nuova1} and \eqref{nuova2} we have also that   $\|\nabla w_{A_{h_i}\cup A^{\eps_i}}\|_{L^p(\Omega)}\to\|\nabla w_{A}\|_{L^p(\Omega)}$, hence $w_{A_{h_i}\cup A^{\eps_i}}$ converges strongly in $W^{1,\,p}_0(\Om)$ to $w_A$. In turn, Theorem~\ref{DM6.3} implies  that 
 $A_{h_i}\cup A^{\eps_{i}}$ $\gamma_p$-converges to $A$. Thus, using the lower semicontinuity of $F$ we conclude that
$$F(A)\leq \liminf_{i\to \infty} F(A_{h_i}\cup A^{\varepsilon_{i}})
\leq \liminf_{i\to \infty} F(A_{h_i})\leq \liminf_{h\to \infty}  J(w_h), $$
thereby finishing the proof.
\end{proof}
Combining  the definition of $G$ with Lemma~\ref{lower semicontinuity}, we conclude that $G$ satisfies \ref{iii}, and hence it is the desired functional. 
\begin{proof}[of Theorem~\ref{thm:existence}]
First, we observe that if $0< c\leq|\Omega|$ the following problem has a solution
$$\inf \{G(w)\colon w\in K,\, |\{w>0\}|\le c\}.$$
Indeed, if $w_h$ is a minimizing sequence, from \eqref{equ 1} it follows that $w_h$ is bounded in $W^{1,\,p}_0(\Om)$. Therefore, up to a not relabelled subsequence, we may assume that $w_h$ converges strongly in $L^p(\Omega)$ and weakly in $W^{1,\,p}_0(\Omega)$ to a function $w_0\in W^{1,\,p}_0(\Omega)$. Since $K$ is closed and convex,  by Hahn- Banach theorem it is also weakly closed, hence $w_0\in K$. Moreover, $|\{w_0>0\}|\leq c$.  Thus, by the lower semicontinuity  property \ref{ii},  we conclude that $w_0$ is a minimizer of the above problem.

Denote now by $A_0$ a $p$-quasi open set such that $\{w_0>0\}\subset A_0\subset\Omega$, with $|A_0|=c$. We claim that $A_0$ is a solution to 
$\min\set{F(A)}{A\in \A_p(\Om),\ |A|\le c}$, hence a solution to the problem \eqref{eq:min prob}.
Indeed, by Lemma~\ref{lem:p subtorsion},  $w_0\le w_{A_0}$ q.e.\ in $\Omega$, and then the properties \ref{i} and \ref{iii} of $G$ imply
$$F(A_0)=G(w_{A_0})\le G(w_0).$$
For any $A\in \mathscr A_p(\Omega)$ and $|A|\le c$, we have $w_A\in K$ and $|\{w_A>0\}|\le c$. Hence the minimality  of $w_0$ yields
$$G(w_0)\le G(w_A)= F(A),$$
which implies $F(A_0)\le F(A)$. By the arbitrariness of $A$ the claim follows and the proof is complete. 
\end{proof}
\begin{Rem}\label{rem5}
{\rm All the statements in this section and in the previous one have been given in the context of $p$-quasi open sets, assuming that $1<p\leq n$. However, all the arguments and tools used in the proofs, including the characterization of the second eigenvalue given by Theorem~\ref{thm:Mpass}, do apply without changes also when $p>n$ and $\mathcal A_p(\Omega)$ reduces to the family of open sets contained in $\Omega$. Therefore both the lower semicontinuity results of Section~4 and Theorem~\ref{thm:existence} still hold in this case.}
\end{Rem}
\section{Appendix}

First, for the reader's convenience, we provide some details 
on the behavior of quasi 
open sets, when restricted to lines parallel to axis. These are useful for 
the purpose of Lemma \ref{lem:quasi res}. 
Let $\Upsilon$ be the set of all compact rectifiable curves 
$\gamma: [0,1] \to \R^n$. Given any family of curves $\Gamma \subset \Upsilon$, the 
$p$-{\it modulus} of the family $\Gamma$ is defined as 
\begin{equation}\label{eq:pmod}
 M_p(\Gamma):= \inf\left\{\int_{\R^n} \rho(x)^p\dx\ |\ \rho:\R^n\to [0,\infty]\ 
 \text{is Borel}, 
 \ \int_\gamma\rho\, ds\geq 1\ \ \forall \gamma\in\Gamma\right\}
\end{equation}
for every $1\leq p<\infty$. 
The notion of $p$-modulus appeared first in \cite{Fuglede1} and later on was extended 
 in the framework of general metric spaces in \cite{Hein-Kos}. It is easy to see that $M_p$ is 
an outer measure on $\Upsilon$.
\begin{Lem}\label{lem:a1}
Let $A\subset\R^n$ be a $p$-quasi open set. Then the set $\inv{\gamma}(A)$ 
 is open in $[0,1]$ for $M_p$-a.e. rectifiable curve $\gamma\in \Upsilon$.
\end{Lem}
 For a proof of the above lemma, we refer to 
\cite[Remark 3.5]{Shan}. 

\begin{Lem}\label{lem:a2}
 Let $E \subset \R^{n-1}$ be a Borel set  and 
$$
  \Gamma_{E} :=\ \{\gamma_{x'}: [0,1] \to E\times [0,1]
 \ |\ x'\in E,\ \gamma_{x'}(t) = (x',t)\}.
$$
If $p>1$, then we have $M_p(\Gamma_E) =0$ if and only if 
  $\mathcal{L}^{n-1}(E) =0$.
\end{Lem}
\begin{proof}
 Let $M_p(\Gamma_{E}) =0$. From the definition \eqref{eq:pmod} for every $m\in \mathbb{N}$ there   exists $\rho_m:\R^n\to [0,\infty]$ such that 
 $\|\rho_m\|_{L^p(\R^n)}^p \leq 1/m $ and $\int_\gamma \rho_m\, ds \geq 1$ for every $\gamma\in \Gamma_{E}$. Hence, if $K$ is a compact subset of $E$, we have 
$$ \mathcal{L}^{n-1}(K)\leq \int_K \int_0^1 \rho_m(x',t)\,dt\,dx'
\leq \mathcal{L}^{n-1}(K)^{1-\frac{1}{p}}\Big(\int_{\R^n} \rho_m(x)^p\dx\Big)^\frac{1}{p}$$
which implies $\mathcal{L}^{n-1}(K)\leq 1/m$  for all $m\in\N$, hence 
$\mathcal{L}^{n-1}(K) =0$. This proves that  $\mathcal{L}^{n-1}(E) =0$.

The converse follows  by observing that if $\mathcal{L}^{n-1}(E) =0$, then the function
$\rho_E = \mathbbm{1}_{E\times [0,1]}$ satisfies  $\int_\gamma\rho_{E}\,ds=1$ for all $\gamma\in\Gamma_{E}$ and  $\|\rho_E\|_{L^p(\R^n)} =0$.
\end{proof}

\begin{Cor}\label{cor:addendum}
Let $A\subset\R^n$ be a $p$-quasi open set. Then 
$A_{x^\prime}:=\set{t\in\R}{(x',t)\in A}$ is an open set for $\mathcal L^{n-1}$-a.e. $x'\in\R^{n-1}$. 
\end{Cor}
\begin{proof}
 Let A be a quasi open set $A\subset\R^n$. Fix  $h\in\mathbb N$ and set
 $$E_h=\left\{x'\in \R^{n-1}\ |\ A\cap\big(\{x'\}\times[-h,h]\big) \,\,\text{is not open in}\,\,\{x'\}\times[-h,h]\right\}.$$ 
 From Lemma \ref{lem:a1}, we have $M_p(\Gamma_{E_h}) =0$, where
 $\Gamma_{E_h}=\{(x^\prime,t): t\in[-h,h]\}$.  Thus from Lemma~\ref{lem:a2} we have $\mathcal{L}^{n-1}(E_h) =0$. Hence, the result  follows.
\end{proof}

Here, we provide the proof of Lemma \ref{lem:p subtorsion}. 

\begin{proof}[of Lemma \ref{lem:p subtorsion}]
Given a closed, convex subset $K \subset W^{1,\,p}_0(\Om)$, we denote by
 $u_K \in K$ a solution of the following variational inequality
 \begin{equation}\label{eq:var ineq}
  \int_{\Omega}|\nabla u_K|^{p-2} 
 \nabla u_K\cdot \nabla (v-u_K)\, dx\geq \int_{\Omega} (v-u_K)\, dx
 \end{equation}
 for all $v\in K$.  Let us now define the operator $\mathcal L: W^{1,\,p}_0(\Om)\to W^{-1,\,p^\prime}\!(\Om)$ setting $\mathcal Lu:=-\Delta_pu$ for all $u\in W^{1,\,p}_0(\Omega)$. Using Lemma~\ref{lem:ineq}, it is immediate to check that $\mathcal L$ satisfies all the monotonicity and coercivity assumptions that guarantee the existence of a solution of the variational inequality \eqref{eq:var ineq}, see  Corollary 1.8 in Ch. III of \cite{Kind-Stamp}.
 
Given $A\in\mathcal A_p(\Omega)$, let us now choose
 $$K:= \set{v\in W^{1,\,p}_0(\Om) }{v\leq 0\ \text{ q.e. in $\Omega\mns A$}}.$$ 
Let $w \in W^{1,\,p}_0(\Om)$ be any subsolution of the equation  $-\lap_p\, w\, = 1 $ in $\Om$ such that $w \leq 0$ q.e. in  $\Om\mns A$. Setting $\varphi:= \min\{u_K-w,0\}\in W^{1,\,p}_0(A)$, using the fact the $w$ is a subsolution, that $\varphi\leq0$  and that $u_K$ satisfies the variational inequality  \eqref{eq:var ineq},  we get
 \begin{align*}
\int_\Om |\gr w|^{p-2}\gr w &\cdot\gr \varphi\,dx
  \geq \int_\Om\varphi\,dx=\int_\Om(u_K-\max\{u_K,w\})\,dx \\
& \geq \int_\Om |\gr u_K|^{p-2}\gr u_K\cdot\gr (u_K-\max\{u_K,w\})\,dx  =\int_\Om |\gr u_K|^{p-2}\gr u_K\cdot\gr \varphi\,dx. 
\end{align*}
In turn, this equivalently can be written as 
$$
\int_{\{u_K<w\}}\big(|\gr w|^{p-2}\gr w-|\gr u_K|^{p-2}\gr u_K)\cdot(\gr w-\gr u_K)\,dx\leq 0.
$$
Thus, from Lemma~\ref{lem:ineq} we conclude that
 $$ \int_{\{u_K<w\}} |\gr u_K-\gr w|^p \dx =0, $$
hence $|\{u_K<w\}| = 0 $ and thus  $ u_K \geq w$ q.e. in $\Om$. 

In particular, taking $ w = 0$, we have that  $u_K\, \geq\, 0$ q.e. in $\Om$. On the other hand  $u_K \in K$, hence $u_K\leq0$ q.e. in $\Omega\setminus A$. Thus, recalling Definition~\ref{def:$p$-quasi sobolev} we have that  $ u_K \in W^{1,\,p}_0(A)$. At this point, choosing as a test function in \eqref{eq:var ineq} $ v = u_K+\psi$  for any  $\psi \in W^{1,\,p}_0(A)$, 
 we get that  $u_K$ is a weak solution of $-\lap_pu=1$ in $A$. By uniqueness, 
 $$u_K=w_A.$$
 In a similar way, choosing $ v = w_A - \psi$ in \eqref{eq:var ineq} 
 for any nonnegative 
 $\psi\in W^{1,\,p}_0(\Om) $ yields that $w_A$ is a  subsolution in $\Om$ of the equation  $-\lap_pu=1$. Hence, the proof is complete.
\end{proof}
We conclude this section by recalling some well known inequalities in the following technical lemma. These have been applied in various places  in the preceeding sections.
For the proof we refer the interested reader to the paper  \cite{Lind1}.

\begin{Lem}\label{lem:ineq} Let $ 1<p<\infty$.  There exists $c(p) >0$ such that
\begin{align*}
(i)\quad&|\xi|^p  - |\eta|^p -p\,|\eta|^{p-2}\eta \cdot(\xi-\eta)\geq  c(p)
\begin{cases}
 |\xi-\eta|^2(|\xi|+|\eta|)^{p-2}\ \ &\text{if}\ \ 1<p<2\\
 |\xi-\eta|^p \ \ &\text{if}\ \ p\geq 2;
\end{cases}\\
(ii)\quad &\left(|\xi|^{p-2}\xi  - |\eta|^{p-2}\eta\right)\cdot(\xi-\eta) \geq c(p)
\begin{cases}
 |\xi-\eta|^2(|\xi|+|\eta|)^{p-2}\ \ &\text{if}\ \ 1<p<2\\
 |\xi-\eta|^p \ \ &\text{if}\ \ p\geq 2.
\end{cases}
\end{align*}
The second inequality of the above implies, in particular, that there exists $c=c(p)>0$ such that for every $u,v\in W^{1,\,p}(\R^n)$ and a measurable set 
$E\subset\R^n$, if $p\geq2$ then 
we have 
 \begin{equation}\label{eq:str mon1}
 \int_{E} |\gr u -\gr v|^p\dx
  \leq c\int_{E}\big( |\gr u|^{p-2}
  \gr u - |\gr v|^{p-2}\gr v\big)\cdot( \gr u-\gr v)\dx,
 \end{equation}
while if $1<p<2$, we have 
 \begin{align}\label{eq:str mon2}
& \int_{E} |\gr u -\gr v|^p\dx
\\
& \qquad\qquad  \leq c\Big(\int_{E}\big( |\gr u|^{p-2}
  \gr u - |\gr v|^{p-2}\gr v\big)\cdot( \gr u-\gr v)\dx\Big)^\frac{p}{2}
  \Big( \int_{E}(|\gr u|+|\gr v|)^p\dx\Big)^{1-\frac{p}{2}}. \nonumber
 \end{align}
\end{Lem}

\vskip 0.6cm

\centerline{\sc Acknowledgments}
\vskip 0.15cm
\noindent
The authors thank T. Kilpel\"ainen and J. Mal\'{y} for 
providing them with appropriate 
references and valuable suggestions. The authors are also thankful to A. Bj\"orn and J. Bj\"orn, who quickly answered a question asked by one of the authors by proving in \cite{BB} the key result stated in Theorem~\ref{thBB}.

%
\end{document}